\documentclass[11pt]{amsart}

\usepackage{amsfonts,epsfig}
\usepackage{latexsym}
\usepackage{amssymb}
\usepackage{amsmath}
\usepackage{amsthm}
\usepackage{graphics}
\usepackage{verbatim}
\usepackage{enumerate}
\usepackage[all]{xy}
\usepackage{array,booktabs,ctable}
\usepackage[backref]{hyperref}
\usepackage{color}
\definecolor{mylinkcolor}{rgb}{0.8,0,0}
\definecolor{myurlcolor}{rgb}{0,0,0.8}
\definecolor{mycitecolor}{rgb}{0,0,0.8}
\hypersetup{colorlinks=true,urlcolor=myurlcolor,citecolor=mycitecolor,linkcolor=mylinkcolor,linktoc=page,breaklinks=true}
\usepackage{bigstrut}

\newtheorem{defn}{Definition}[section]

\newtheorem{corollary}[defn]{Corollary}
\newtheorem{lemma}[defn]{Lemma}

\newtheorem{thm}[defn]{Theorem}
\newtheorem{theorem}[defn]{Theorem}

\newtheorem{prop}[defn]{Proposition}
\newtheorem{proposition}[defn]{Proposition}

\theoremstyle{definition}

\newtheorem*{ack}{Acknowledgements}
\newtheorem{remark}[defn]{Remark}

\newtheorem{example}[defn]{Example}

\newcommand{\Q}{\mathbb Q}
\newcommand{\Qbar}{\overline{\Q}}
\newcommand{\Z}{\mathbb Z}

\newcommand{\F}{\mathbb F}

\newcommand{\SL}{\operatorname{SL}}
\newcommand{\PSL}{\operatorname{PSL}}

\newcommand{\Gal}{\operatorname{Gal}}
\newcommand{\Aut}{\operatorname{Aut}}

\newcommand{\GL}{\operatorname{GL}}
\newcommand{\PGL}{\operatorname{PGL}}

\newcommand{\im}{\mathrm{im}\,}

\newcommand{\tor}{\mathrm{tors}}

\begin{document}



\bibliographystyle{plain}
\title[Torsion over $\Q(3^\infty)$]{Torsion subgroups of rational elliptic curves over the compositum of all cubic fields}

\author{Harris B. Daniels}
\address{Department of Mathematics and Statistics, Amherst College, Amherst, MA 01002, USA}
\email{hdaniels@amherst.edu}
\urladdr{http://www3.amherst.edu/~hdaniels/}

\author{\'Alvaro Lozano-Robledo}
\address{Dept. of Mathematics, Univ. of Connecticut, Storrs, CT 06269, USA}
\email{alvaro.lozano-robledo@uconn.edu}
\urladdr{http://alozano.clas.uconn.edu/}

\author{Filip Najman}
\address{Department of Mathematics, University of Zagreb, Bijeni\v{c}ka cesta 30, 10000 Zagreb, CROATIA}
\email{fnajman@math.hr}
\urladdr{http://web.math.pmf.unizg.hr/~fnajman/}

\author{Andrew V. Sutherland}
\address{Department of Mathematics, MIT, Cambridge, MA 02139, USA}
\email{drew@math.mit.edu}
\urladdr{http://math.mit.edu/~drew/}

\thanks{The third author acknowledges
	support from the QuantiXLie Center of Excellence. The fourth author was supported by NSF grants DMS-1115455 and DMS-1522526.}



\begin{abstract} Let $E/\Q$ be an elliptic curve and let $\Q(3^\infty)$ be the compositum of all cubic extensions of $\Q$. In this article we show that the torsion subgroup of $E(\Q(3^\infty))$ is finite and determine 20 possibilities for its structure, along with a complete description of the $\Qbar$-isomorphism classes of elliptic curves that fall into each case.
We provide rational parameterizations for each of the 16 torsion structures that occur for infinitely many $\Qbar$-isomorphism classes of elliptic curves, and a complete list of $j$-invariants for each of the 4 that do not.
\end{abstract}

\maketitle

\section{Introduction}

Interest in the rational points on elliptic curves dates back at least to Poincar\'e, who in 1901 conjectured that the group $E(\Q)$ of rational points on an elliptic curve $E$ over $\Q$ is a finitely generated abelian group \cite{poincare}.
This conjecture was proved by Mordell \cite{mordell} in 1922 and then vastly generalized by Weil \cite{weil}, who proved in 1929 that the group of rational points on an abelian variety defined over a number field is finitely generated.
An immediate consequence of the Mordell-Weil theorem is that the torsion subgroup $E(F)_\tor$ of an elliptic curve $E$ over a number field $F$ is finite, and therefore isomorphic to a group of the form
\[
\Z/a\Z \oplus \Z/ab\Z,
\]
for some integers $a,b\geq 1$. In 1996, Merel \cite{merel} proved the existence of a uniform bound on the cardinality of $E(F)_\tor$ that depends only on the number field $F$, not the particular elliptic curve $E/F$; in fact, Merel's bound depends only on the degree of the field extension $F/\Q$.
This bound was improved and made effective by Oesterl\'e in 1994 (unpublished), and later by Parent \cite{parent} in 1999.

It is thus a natural goal to classify (up to isomorphism), the torsion subgroups of elliptic curves defined over number fields of degree $d$, for fixed integers $d\geq 1$. Mazur famously proved such a classification for $d=1$.

\begin{thm}[Mazur \cite{mazur1}]\label{thm-mazur}
Let $E/\Q$ be an elliptic curve. Then
\[
E(\Q)_\tor\simeq
\begin{cases}
\Z/M\Z &\text{with}\ 1\leq M\leq 10\ \text{or}\ M=12,\ \text{or}\\
\Z/2\Z\oplus \Z/2M\Z &\text{with}\ 1\leq M\leq 4.
\end{cases}
\]
\end{thm}

The classification for $d=2$ was initiated by Kenku and Momose, and completed by Kamienny.

\begin{thm}[Kenku, Momose \cite{kenmom}, Kamienny \cite{kamienny}]\label{thm-quadgroups}
Let $E/F$ be an elliptic curve over a quadratic number field $F$. Then
\[
E(F)_\tor\simeq
\begin{cases}
\Z/M\Z &\text{with}\ 1\leq M\leq 16\ \text{or}\ M=18,\ \text{or}\\
\Z/2\Z\oplus \Z/2M\Z &\text{with}\ 1\leq M\leq 6,\ \text{or}\\
\Z/3\Z \oplus \Z/3M\Z &\text{with}\ M=1\ \text{or}\ 2,\ \text{only if}\ F = \Q(\sqrt{-3}),\ \text{or}\\
\Z/4\Z \oplus \Z/4\Z &\text{only if}\ F = \Q(\sqrt{-1}).
\end{cases}
\]
\end{thm}

The case $d=3$ remains open. Jeon, Kim, and Schweizer have determined the torsion structures that appear infinitely often as one runs through all elliptic curves over all cubic fields \cite{jeon2}, and Jeon, Kim, and Lee have constructed infinite families of elliptic curves that realize each of these torsion structures~\cite{jeon1}.

\begin{thm}[Jeon, Kim, Lee, Schweizer \cite{jeon1,jeon2}] Suppose that $T$ is an abelian group for which there exist infinitely many $\Qbar$-isomorphism classes of elliptic curves $E$ over cubic number fields $F$, such that $E(F)_\tor\simeq T$. Then
\begin{align*}
T\simeq\begin{cases}
\Z/M\Z & \text{with } 1\leq M \leq 16 \text{ or } M = 18,20,\ \text{or}\\
\Z/2\Z\oplus \Z/2M\Z &\text{with } 1\leq M \leq 7.
\end{cases}
\end{align*}
Moreover, for each such $T$ an explicit infinite family of elliptic curves over cubic fields with torsion subgroup isomorphic to $T$ is known that contains infinitely many $\Qbar$-isomorphism classes.
\end{thm}

A similar list of possible torsion structures that appear infinitely often as one runs through all elliptic curves over all quartic fields was determined by Jeon, Kim and Park \cite{jeon3} and infinite families of elliptic curves that realize each of these torsion structures were constructed by Jeon, Kim, and Lee \cite{jeon4}.

Sharper results can be proved if one restricts to base extensions of elliptic curves that are defined over $\Q$. In this setting the second author has obtained bounds on the largest prime-power order that may appear in a torsion subgroup \cite{lozano1, lozano5}, and the third author has classified the torsion subgroups that can arise over extensions of degrees 2 and 3 \cite{najman}. Chou \cite{chou} has classified the groups that can appear when base-extending to a Galois quartic extension of $\Q$.

\begin{thm}\cite[Thm.~2]{najman}
\label{thm-najman1} Let $E/\Q$ be an elliptic curve and let $F$ be a quadratic number field. Then
\[
E(F)_\tor\simeq
\begin{cases}
\Z/M\Z &\text{with}\ 1\leq M\leq 10 \text{ or } M = 12,15,16,\ \text{or}\\
\Z/2\Z\oplus \Z/2M\Z &\text{with}\ 1\leq M\leq 6,\ \text{or}\\
\Z/3\Z \oplus \Z/3M\Z &\text{with}\ 1\leq M\leq 2 \text{ and }F = \Q(\sqrt{-3}),\ \text{or}\\
\Z/4\Z \oplus \Z/4\Z &\text{with}\ F = \Q(\sqrt{-1}).
\end{cases}
\]

\end{thm}

\begin{thm}\cite[Thm.~1]{najman}
\label{thm-najman2} Let $E/\Q$ be an elliptic curve and let $F$ be a cubic number field. Then
\[
E(F)_\tor\simeq
\begin{cases}
\Z/M\Z &\text{with}\ 1\leq M\leq 10 \text{ or } M = 12,13,14,18,21,\  \text{or}\\
\Z/2\Z\oplus \Z/2M\Z &\text{with}\ 1\leq M\leq 4 \text{ or } M=7.
\end{cases}
\]
Moreover, the elliptic curve $162B1$ over $\Q(\zeta_9)^+$ is the unique rational elliptic curve over a cubic field with torsion subgroup isomorphic to $\Z/21\Z$. For all other groups $T$ listed above there are infinitely many $\Qbar$-isomorphism classes of elliptic curves $E/\Q$ for which $E(F)\simeq T$ for some cubic field $F$.
\end{thm}

In the setting of base extensions of elliptic curves $E/\Q$, one may also consider the torsion subgroups that can arise over certain infinite algebraic extensions of $\Q$.  In general these need not be finite, and there may be infinitely many possibilities; but for suitably chosen extensions, this is not the case.
For example, Ribet proved that for an abelian variety defined over a number field $F$, the torsion subgroup of its base change to the maximal cyclotomic extension of $F$ is finite \cite{ribet}.
Here we consider infinite extensions obtained as the compositum of all number fields of a fixed degree~$d$.

\begin{defn}
For each fixed integer $d\ge 1$, let $\Q(d^\infty)$ denote the compositum of all field extensions $F/\Q$ of degree $d$.
More precisely, let $\Qbar$ be a fixed algebraic closure of $\Q$, and define
\[
\Q(d^\infty) := \Q\bigl(\{\beta\in \Qbar: [\Q(\beta):\Q]=d  \}\bigr).
\]
\end{defn}

The fields $\Q(d^\infty)$ have been studied by Gal and Grizzard \cite{grizzard}, who use the notation $\Q^{[d]}$ (they also consider the fields $\Q^{(d)}=\Q^{[2]}\Q^{[3]}\cdots\Q^{[d]}$ and show that $\Q^{[d]}=\Q^{(d)}$ precisely when $d<5$).
For elliptic curves $E/\Q$, the group $E(\Q(d^\infty))$ is not finitely generated.
This was proved for $d=2$ by Frey and Jarden \cite{fj} in 1974, and the result for $d\ge 2$ follows from the inclusion $\Q(2^\infty)\subseteq \Q(d^\infty)$ given by \cite[Theorem 1]{grizzard}.

The torsion subgroups of $E(\Q(d^\infty))$ have been studied in the case $d=2$, in which the field $\Q(2^\infty)$ is the maximal elementary abelian $2$-extension of $\Q$.
Even though $E(\Q(2^\infty))$ is not finitely generated, the torsion subgroup $E(\Q(2^\infty))_\tor$ is known to be finite, and the possible torsion structures have been classified by  Laska and Lorenz \cite{laska}, and Fujita \cite{fujita1,fujita2}.

\begin{thm}[Laska, Lorenz \cite{laska}, Fujita \cite{fujita1,fujita2}]\label{thm-fujita} Let $E/\Q$ be an elliptic curve and let
\[
\Q(2^\infty) := \Q\bigl(\{\sqrt{m}: m\in\Z\}\bigr).
\]
The torsion subgroup $E(\Q(2^\infty))_\tor$ is finite, and
\[
E(\Q(2^\infty))_\tor\simeq
\begin{cases}
\Z/M\Z &\text{with}\ M\in 1,3,5,7,9,15,\ \text{or}\\
\Z/2\Z\oplus \Z/2M\Z &\text{with}\ 1\leq M\leq 6\ \text{or}\ M=8,\ \text{or}\\
\Z/3\Z\oplus \Z/3\Z & \text{or}\\
\Z/4\Z\oplus \Z/4M\Z &\text{with}\ 1\leq M\leq 4,\ \text{or}\\
\Z/2M\Z\oplus \Z/2M\Z &\text{with}\ 3\leq M\leq 4.\\
\end{cases}
\]
\end{thm}

In this article we classify the torsion subgroups $E(\Q(3^\infty))_\tor$ that arise for elliptic curves $E/\Q$.
Our main theorem is the following.

\begin{thm}\label{thm-main} Let $E/\Q$ be an elliptic curve. The torsion subgroup $E(\Q(3^\infty))_\tor$ is finite, and
\[
E(\Q(3^\infty))_\tor\simeq
\begin{cases}
\Z/2\Z\oplus \Z/2M\Z &\text{with}\ M=1,2,4,5,7,8,13, \text{ or}\\
\Z/4\Z\oplus \Z/4M\Z &\text{with}\ M=1,2,4,7,\ \text{or}\\
\Z/6\Z\oplus \Z/6M\Z &\text{with}\ M=1,2,3,5,7,\ \text{or}\\
\Z/2M\Z\oplus \Z/2M\Z &\text{with}\ M=4,6,7,9.\\
\end{cases}
\]
All but $4$ of the torsion subgroups $T$ listed above occur for infinitely many $\Qbar$-isomorphism classes of elliptic curves $E/\Q$; for $T= \Z/4\Z\times\Z/28\Z,\ \Z/6\Z\times\Z/30\Z,\ \Z/6\Z\times\Z/42\Z$, and $\Z/14\Z\times\Z/14\Z$ there are only $2$, $2$, $4$, and $1$ (respectively) $\Qbar$-isomorphism classes of $E/\Q$ for which $E(\Q(3^\infty))_\tor\simeq T$.

\end{thm}

\begin{remark}\label{rem-examples}
Minimal conductor examples of elliptic curves $E/\Q$ that realize each of the torsion subgroups permitted by Theorem \ref{thm-main} are listed in the table below.
Here and throughout we identify elliptic curves over $\Q$ by their Cremona label \cite{cremona} and provide a hyperlink to the corresponding entry in the $L$-functions and Modular Forms Database (LMFDB) \cite{lmfdb}.
\smallskip

\begin{center}
\begin{tabular}{l|l||l|l}
	$E/\Q$ & $E(\Q(3^\infty))_\tor$ & $E/\Q$ & $E(\Q(3^\infty))_\tor$\\
	\hline\bigstrut[t]
	\href{http://www.lmfdb.org/EllipticCurve/Q/11a2}{\texttt{11a2}}     & $\Z/2\Z\oplus \Z/2\Z$ &
    \href{http://www.lmfdb.org/EllipticCurve/Q/338a1}{\texttt{338a1}}   & $\Z/4\Z\oplus \Z/28\Z$ \\
	\href{http://www.lmfdb.org/EllipticCurve/Q/17a3}{\texttt{17a3}}     & $\Z/2\Z\oplus \Z/4\Z$ &
    \href{http://www.lmfdb.org/EllipticCurve/Q/20a1}{\texttt{20a1}}     & $\Z/6\Z\oplus \Z/6\Z$ \\
    \href{http://www.lmfdb.org/EllipticCurve/Q/15a5}{\texttt{15a5}}     & $\Z/2\Z\oplus \Z/8\Z$ &
    \href{http://www.lmfdb.org/EllipticCurve/Q/30a1}{\texttt{30a1}}     & $\Z/6\Z\oplus \Z/12\Z$ \\
	\href{http://www.lmfdb.org/EllipticCurve/Q/11a1}{\texttt{11a1}}     & $\Z/2\Z\oplus \Z/10\Z$ &
	\href{http://www.lmfdb.org/EllipticCurve/Q/14a3}{\texttt{14a3}}     & $\Z/6\Z\oplus \Z/18\Z$ \\
	\href{http://www.lmfdb.org/EllipticCurve/Q/26b1}{\texttt{26b1}}     & $\Z/2\Z\oplus \Z/14\Z$ &
	\href{http://www.lmfdb.org/EllipticCurve/Q/50a3}{\texttt{50a3}}     & $\Z/6\Z\oplus \Z/30\Z$ \\
	\href{http://www.lmfdb.org/EllipticCurve/Q/210e1}{\texttt{210e1}}   & $\Z/2\Z\oplus \Z/16\Z$ &
	\href{http://www.lmfdb.org/EllipticCurve/Q/162b1}{\texttt{162b1}}   & $\Z/6\Z\oplus \Z/42\Z$ \\
	\href{http://www.lmfdb.org/EllipticCurve/Q/147b1}{\texttt{147b1}}   & $\Z/2\Z\oplus \Z/26\Z$ &
	\href{http://www.lmfdb.org/EllipticCurve/Q/15a1}{\texttt{15a1}}     & $\Z/8\Z\oplus \Z/8\Z$ \\
	\href{http://www.lmfdb.org/EllipticCurve/Q/17a1}{\texttt{17a1}}     & $\Z/4\Z\oplus \Z/4\Z$ &
	\href{http://www.lmfdb.org/EllipticCurve/Q/30a2}{\texttt{30a2}}     & $\Z/12\Z\oplus \Z/12\Z$ \\
	\href{http://www.lmfdb.org/EllipticCurve/Q/15a2}{\texttt{15a2}}     & $\Z/4\Z\oplus \Z/8\Z$ &
    \href{http://www.lmfdb.org/EllipticCurve/Q/2450a1}{\texttt{2450a1}} & $\Z/14\Z\oplus \Z/14\Z$ \\
	\href{http://www.lmfdb.org/EllipticCurve/Q/210e2}{\texttt{210e2}}   & $\Z/4\Z\oplus \Z/16\Z$ &
	\href{http://www.lmfdb.org/EllipticCurve/Q/14a1}{\texttt{14a1}}     & $\Z/18\Z\oplus \Z/18\Z$ \\
	\hline
\end{tabular}
\end{center}
\bigskip

Magma \cite{magma} scripts to verify these examples, and all other computational results cited herein, are available at~\cite{magmascripts}.
These include explicit models of the modular curves we constructed in the course of proving our theorems, two algorithms to compute $E(\Q(3^\infty))_\tor$ for an elliptic curve $E/\Q$ (one is described in \S\ref{sec:algorithm} and the other is an effective version of Theorem~\ref{thm-param}), and an implementation of the computational strategy that is used to prove Theorem~\ref{thm-param}, which precisely characterizes the sets of elliptic curves that realize each of the subgroups listed in Theorem~\ref{thm-main}, and in particular, which are finite and which are infinite.
\end{remark}

For each of the torsion structures $T$ in Theorem~\ref{thm-main} that arises infinitely often we provide a complete set of rational functions that parameterize the $j$-invariants of the elliptic curves $E/\Q$ for which $E(\Q(3^\infty))_\tor$ contains a subgroup isomorphic to $T$ (for the general member of each family, isomorphism holds), and for those that occur only finitely often we provide a complete list of $j$-invariants; this information appears in Table~\ref{table-param} at the end of the article.

Key to our results are a number of recent advances in our explicit understanding of Galois representations attached to elliptic curves over number fields.
In particular, we rely on work of Rouse and Zureick-Brown \cite{RZB} classifying the 2-adic representations of elliptic curves over $\Q$, work of Zywina~\cite{zywina1} on the possible mod-$p$ representations of an elliptic curve over~$\Q$,
and algorithms developed by the fourth author \cite{sutherland2} for efficiently computing the images of Galois representations of elliptic curves over number fields.

\begin{ack}
The authors would like to thank Robert Grizzard for helpful conversations about the structure of $\Gal(\Q(3^\infty)/ \Q)$ and Jackson Morrow, Jeremy Rouse, David Zureick-Brown, and David Zywina, for their computational assistance, including explicit models for some of the modular curves that appear in this article.
We thank Lukas Pottmeyer and David Zureick-Brown for their feedback on an early draft of this article, and we thank Maarten Derickx for pointing out an error in the proof of Lemma 3.2 that is corrected in this version.
\end{ack}

\section{Notation and terminology}\label{sec-notation}
We fix once and for all an algebraic closure $\Qbar$ that contains all the algebraic extensions of $\Q$ that we may consider, including the fields $\Q(d^\infty)$ and the Galois closure and algebraic closure of every number field.
As usual, for an elliptic curve $E/F$, we use $E[n]$ to denote the $n$-torsion subgroup of $E(\overline{F})$, where $\overline{F}=\Qbar$ when $F$ is a number field.
We recall that $E[n]\simeq \Z/n\Z\oplus\Z/n\Z$, so long as $n$ is prime to the characteristic of $F$, which holds for all the cases we consider.
If $L/F$ is a field extension, we write $E(L)[n]$ for the $n$-torsion subgroup of $E(L)$, and for primes $p$, we write $E(L)(p)$ for the $p$-primary component of $E(L)$.
For any point or set of points $\mathcal{P}$ in $E(\overline{F})$, we write $F(\mathcal{P})$ for the extension generated by the coordinates of $\mathcal{P}$ and $F(x(\mathcal{P}))$ for the extension generated by the $x$-coordinates of $\mathcal{P}$ (we assume $E$ is given by a Weierstrass equation in $x$ and $y$).

For an elliptic curve $E/F$, an $n$-\emph{isogeny} is a cyclic isogeny $\varphi\colon E\to E'$ of degree $n$; this means $\ker\varphi$ is a cyclic subgroup of $E[n]$, and as all the isogenies we consider are separable, this cyclic group has order $n$.
The isogenies $\varphi$ that we consider are also \emph{rational}, meaning that $\varphi$ is defined over~$F$, equivalently, that $\ker\varphi$ is \emph{Galois-stable}: the action of $\Gal(\overline{F}/F)$ on $E[n]$ given by its action on the coordinates of the points $P\in E[n]$ permutes $\ker\varphi\subseteq E[n]$.
To avoid any possible confusion, we will usually state the rationality of $\varphi$ explicitly.
We consider two (separable) isogenies to be distinct only when their kernels are distinct (otherwise they differ only by an isomorphism).

We recall that if $E/\Q$ is an elliptic curve, then for each positive integer $n$ the action of the group $\Gal(\Qbar/\Q)$ on the $\Z/n\Z$-module $E[n]$ induces a \emph{Galois representation} (continuous homomorphism)
\[
\rho_{E,n}\colon \Gal(\Qbar/\Q)\to \textrm{Aut}(E[n])\simeq \GL_2(\Z/n\Z),
\]
whose image we view as a subgroup of $\GL_2(\Z/n\Z)$ (determined only up to conjugacy).
When $n=p$ is prime, we may identify $\GL_2(\Z/p\Z)$ with $\GL_2(\F_p)$.
The extension $\Q(E[n])/\Q$ is Galois, and the homomorphism $\Gal(\Q(E[n])/\Q)\to\GL_2(\Z/n\Z)$ induced by restriction is injective; thus $\Gal(\Q(E[n])/\Q)$ is isomorphic to a subgroup of $\GL_2(\Z/n\Z)$.
This subgroup necessarily contains elements of every possible determinant (each residue class in $(\Z/n\Z)^\times$ contains the norms of infinitely many unramified primes of $\Q(E[n])/\Q$), and an element $\gamma$ with trace $0$ and determinant $-1$ (corresponding to complex conjugation).\footnote{The element $\gamma$ also must act trivially on a maximal cyclic subgroup of $\Z/n\Z\oplus\Z/n\Z$ corresponding to the real line, an additional constraint that is important when $n$ is even; see Remark 3.14 in \cite{sutherland2}.}
We refer the reader to \cite{serre} for further background on Galois representations.

We distinguish two standard subgroups of $\GL_2(\Z/n\Z)$ (up to conjugacy): (1) the \emph{Borel} group of upper triangular matrices, and (2) the \emph{split Cartan} group of diagonal matrices.
Recall that an elliptic curve $E/\Q$ admits a rational $n$-isogeny if and only if the image of $\rho_{E,n}$ in $\GL_2(\Z/n\Z)$ is conjugate to a subgroup of the Borel group (both conditions hold if and only if $E[n]$ contain a Galois-stable cyclic subgroup of order~$n$).
Similarly, $E/\Q$ admits two rational $n$-isogenies whose kernels intersect trivially if and only if the image of $\rho_{E,n}$ in $\GL_2(\Z/n\Z)$ is conjugate to a subgroup of the split Cartan group.

If $H$ is a subgroup of $\GL_2(\Z/n\Z)$ with surjective determinant map that contains $-1$, we use $X_H$ to denote the associated modular curve over $\Q$ whose non-cuspidal rational points parameterize elliptic curves $E/\Q$ for which the image of $\rho_{E,n}$ in $\GL_2(\Z/n\Z)$ is conjugate to a subgroup of $H$.
Certain information about $X_H$, including its genus, can be determined from the congruence subgroup $\Gamma_H$ of $\PSL_2(\Z)$ obtained by taking the inverse image of the intersection of $H$ with $\SL_2(\Z/n\Z)$ in $\PSL_2(\Z)=\SL_2(\Z)/\{\pm 1\}$.
The tables of Cummins and Pauli \cite{cp} contain data for all congruence subgroups of genus up to 24 in which subgroups are identified by labels of the of the form ``$m\textrm{X}^g$", where $m$ is the level, $g$ is the genus, and X is a letter that distinguishes groups of the same level and genus.  We note that the level $m$ of $\Gamma_H$ divides but need not equal $n$, and two non-conjugate $H_1$ and $H_2$ may give rise to the same congruence subgroup $\Gamma_{H_1}=\Gamma_{H_2}$ in $\PSL_2(\Z)$.

\section{The field \texorpdfstring{$\Q(3^\infty)$}{\textit{Q}(3*)}}

As noted in the introduction, the field $\Q(2^\infty)\subseteq \Q(3^\infty)$ is the maximal elementary abelian 2-extension of $\Q$; the number fields in $\Q(2^\infty)$ are precisely those whose Galois group is isomorphic to $(\Z/2\Z)^n$ for some integer $n\ge 0$.
In this section we similarly characterize the number fields in $\Q(3^\infty)$ in terms of their Galois groups.

\begin{defn}\label{def-gens3}
We say that a finite group $G$ is of \textbf{generalized} $\boldsymbol{S_3}$-\textbf{type}, if it is isomorphic to a subgroup of a direct product $S_3\times\cdots\times S_3$ of symmetric groups of degree $3$.
\end{defn}

Recall that a finite group $G$ is \emph{supersolvable} (or supersoluble) if it has a normal cyclic series; an equivalent criterion is that every maximal subgroup of $G$ has prime index \cite{huppert}, or that every subgroup of $G$ is Lagrangian (each subgroup $H$ contains subgroups of every order dividing $|H|$) \cite{zappa}.
The following lemma characterizes finite groups of generalized $S_3$-type.

\begin{lemma}\label{lem-s3char}
A finite group $G$ is of generalized $S_3$-type if and only if (i) $G$ is supersolvable, (ii) the exponent of $G$ divides $6$, and (iii) the Sylow subgroups of $G$ are abelian.
\end{lemma}
\begin{proof}
For the forward implication, properties (i), (ii), and (iii) are preserved by direct products and subgroups (and quotients).
Thus to show that every finite group $G$ of generalized $S_3$-type has all three properties, it is enough to note that $S_3$ does, which is clearly the case.

For the reverse implication, let $G$ be a finite group satisfying (i), (ii), and (iii).
Property (i) implies that $G$ admits a cyclic normal series whose successive quotients have non-increasing prime orders (see \cite[Thm.\ 5.4.8]{robinson}, for example), and property (ii) then implies that the 3-Sylow subgroup~$V$ is normal.
By properties (ii) and (iii), $V$ is an elementary abelian $3$-group that we may view as an $\F_3$-vector space equipped with an action (via conjugation) of the $2$-sylow subgroup~$W$, an elementary abelian $2$-group; in other words, $G\simeq V\rtimes W\simeq (\Z/3\Z)^m\rtimes (\Z/2\Z)^n$, for some $m,n\ge 0$.

Let $\tau_1,\ldots,\tau_n$ be generators for $W$.
As linear operators on $V$, the $\tau_i$ are diagonalizable (their minimal polynomials $T^2-1\in \F_3[T]$ split into distinct linear factors), and simultaneously diagonalizable, since they commute.  We may thus choose generators $\sigma_1,\ldots,\sigma_m$ for $V$ such that $\tau_i\sigma_j\tau_i=\sigma_j^{\pm 1}$.

Let us now enlarge $G$ to $G'$ by adding generators $\tau_{n+1},\ldots, \tau_{n+m}$ of order 2 that commute with all the $\sigma_i$ and~$\tau_j$ except for the relations $\tau_{n+i}\sigma_j\tau_{n+i}=\sigma_j^{-1}$.  For $1\le i\le n$ let us replace each $\tau_i$ with $\tau_i\tau_{n+j_1}\cdots \tau_{n+j_k}$, where $j_1,\ldots, j_k$ are precisely the indices $j$ for which $\tau_i\sigma_j\tau_i=\sigma_j^{-1}$.  Our new $\tau_1,\ldots,\tau_n$ then commute with all the $\sigma_j$, and we have
\[
G\subseteq G'\simeq \langle\tau_1,\ldots,\tau_n\rangle\times \langle \sigma_1,\tau_{n+1}\rangle\times\cdots\times\langle \sigma_m,\tau_{n+m}\rangle\simeq (\Z/2\Z)^n\times S_3^m\subseteq S_3^{n+m},
\]
so $G$ is of generalized $S_3$-type as claimed.
\end{proof}

\begin{example}
The alternating group $A_4$, the cyclic group $\Z/4\Z$, and the Burnside group $B(2,3)$ (the unique non-abelian group of order 27 and exponent 3) are examples of groups that are not of generalized $S_3$-type; each satisfies only two of the three properties required by Lemma~\ref{lem-s3char}.
\end{example}

\begin{corollary}\label{cor-s3quo}
The product of two groups of generalized $S_3$-type is of generalized $S_3$-type, as is every subgroup and every quotient of a group of generalized $S_3$-type.
\end{corollary}

Our main goal in this section is to show that the groups that arise as Galois groups of number fields in $\Q(3^\infty)$ are precisely the groups of generalized $S_3$-type.
We first address the forward implication.

\begin{thm}\label{thm-galq3}
Let $L$ be a number field in $\Q(3^\infty)$ with Galois closure $\widehat{L}$.
Then $\widehat{L}\subseteq \Q(3^\infty)$ and
$\Gal(\widehat{L}/\Q)$ is of generalized $S_3$-type. In particular, the exponent of $\Gal(\widehat{L}/\Q)$ divides $6$.
\end{thm}

\begin{proof}
Every number field $L$ in $\Q(3^\infty)$ lies in a compositum of cubic fields $F_1\cdots F_m$.
The compositum of the Galois closures $\widehat{F}_1\cdots \widehat{F}_m$ is a Galois extension $\widehat{F}/\Q$ that contains~$L$, and therefore~$\widehat{L}$, and it is a subfield of $\Q(3^\infty)$, since we can write each $\widehat{F}_i =F_{i,1}F_{i,2}F_{i,3}$ as a compositum of cubic fields $F_{i,j}:=\Q(\alpha_j)$ generated by the roots $\alpha_j$ of an irreducible cubic polynomial defining~$F_i/\Q$.
Each $G_i:=\Gal(\widehat F_i)$ is isomorphic to either $\Z/3\Z$ or $S_3$, both of which are of generalized $S_3$-type, $\Gal(\widehat{F}/\Q)$ is isomorphic to a subgroup of $G_1\times\cdots\times G_m$, hence of generalized $S_3$-type, and $\Gal(\widehat{L}/\Q)$ is isomorphic to a quotient of $\Gal(\widehat{F}/\Q)$, hence also of generalized $S_3$-type, by Corollary~\ref{cor-s3quo}.
\end{proof}

We now prove the converse of Theorem \ref{thm-galq3}.

\begin{thm}\label{thm-galq3conv} Let $L$ be a number field with Galois closure $\widehat{L}$ such that $\Gal(\widehat{L}/\Q)$ is of generalized $S_3$-type. Then $L\subseteq \widehat{L}\subseteq \Q(3^\infty)$. Indeed, $\widehat L$ is the compositum of its quadratic and cubic subfields.
\end{thm}
\begin{proof}
Let $G=\Gal(\widehat{L}/\Q)$, and let $\tau_1,\ldots,\tau_n$ be independent generators for its 2-Sylow subgroup.
As explained in the proof of Lemma~\ref{lem-s3char}, we can choose independent generators $\sigma_1,\ldots,\sigma_m$ for the 3-Sylow subgroup $H$ of $G$ such that $\tau_i\sigma_j\tau_i=\sigma_j^{\pm 1}$ for all $i,j$.  For $1\le j\le m$, let
\[
H_j:=\langle \sigma_1,\ldots,\hat\sigma_j,\ldots,\sigma_m,\tau_1,\ldots,\tau_n\rangle,
\]
where the notation $\hat\sigma_j$ indicates that $\sigma_j$ is excluded.
The fixed field $\widehat L^H$ is equal to the compositum of all the quadratic subfields of $\widehat L$, and each of the fixed fields $\widehat L^{H_j}$ is a cubic subfield of $\widehat L$.
The intersection $H\cap H_1\cap\cdots\cap H_m$ is trivial, and it follows that
\[
L\subseteq \widehat L = \widehat L^H\widehat L^{H_1}\cdots\widehat L^{H_m}\subseteq\Q(2^\infty)\cdot\Q(3^\infty)\subseteq \Q(3^\infty),
\]
since $\Q(2^\infty)\subseteq \Q(3^\infty)$, and that $\widehat L$ is the compositum of its quadratic and cubic subfields.
\end{proof}

We will appeal to Theorems~\ref{thm-galq3} and \ref{thm-galq3conv} repeatedly in the sections that follow; for the sake of brevity we do not cite them in every case.

We conclude this section by determining the roots of unity $\zeta_n$ of prime-power order $n$ that lie in~$\Q(3^\infty)$.
The possible values of $n$ are severely constrained by the fact that if $\zeta_n\in \Q(3^\infty)$, then the exponent of $\Gal(\Q(\zeta_n)/\Q)\simeq (\Z/n\Z)^\times$ must divide 6.

\begin{lemma}\label{lem-roots_of_unity}
Let $n$ be a prime power.  Then $\Q(\zeta_n)\subseteq\Q(3^\infty)$ if and only if $n\in \{2,3,4,7,8,9\}$.
\end{lemma}
\begin{proof}
Suppose $\Q(\zeta_n)\subseteq \Q(3^\infty)$.
Then the exponent $\lambda(n)$ of $\Gal(\Q(\zeta_n)/\Q)\simeq(\Z/n\Z)^\times$ divides 6.
We have $\lambda(2^e)=2^{e-2}$ and $\lambda(p^e)=\varphi(p^e)=(p-1)p^{e-1}$ for primes $p>2$.
It follows that $\lambda(n)$ divides 6 only for $n\in \{2,3,4,7,8,9\}$.
The group $(\Z/n\Z)^\times$ is abelian, hence it is supersolvable and has abelian Sylow subgroups. Lemma~\ref{lem-s3char} and Theorem~\ref{thm-galq3conv} imply $\Q(\zeta_n)\subseteq\Q(3^\infty)$ for $n\in \{2,3,4,7,8,9\}$.
\end{proof}

\section{Finiteness Results}

Our goal in this section is to prove that $E(\Q(3^\infty))_\tor$ is finite.
The only property of $\Q(3^\infty)$ that we actually require is that it is a Galois extension of $\Q$ that contains only a finite number of roots of unity, a property that applies to all the fields $\Q(d^\infty)$.
We thus work in a more general setting.

\begin{thm}\label{thm-finite_over_F}
Let $E/\Q$ be an elliptic curve and let $F$ be a (possibly infinite) Galois extension of~$\Q$ that contains only finitely many roots of unity. Then $E(F)_\tor$ is finite.  Moreover, there is a uniform bound $B$, depending only on $F$, such that $\#E(F)_\tor\le B$ for every elliptic curve $E/\Q$.
\end{thm}

Before proving the theorem we first establish some intermediate results.
We begin with the usual consequence of the existence of the Weil pairing.

\begin{prop}\cite[Ch. III, Cor. 8.1.1]{silverman} \label{prop-weil} Let $E/L$ be an elliptic curve with $L\subseteq\Qbar$. For each integer $n\geq 1$, if $E[n]\subseteq E(L)$ then the $n$th cyclotomic field $\Q(\zeta_n)$ is a subfield of $L$.
\end{prop}

This immediately implies the following result.

\begin{lemma}\label{lem-finite_full_tors}
Let $E$ and $F$ be as in Theorem~\ref{thm-finite_over_F}.  Then $E[n]\subseteq E(F)$ for only finitely many $n$.
\end{lemma}

The following theorem summarizes results of Mazur and Kenku that yield a complete classification of the rational $n$-isogenies that can arise for elliptic curves over $\Q$ (recall that $n$-isogenies are defined to be cyclic).  See \cite[\S9]{lozano1} for further details.

\begin{theorem}\label{thm-isog}\cite{mazur2,kenku2,kenku3,kenku4,kenku5}
Let $E/\Q$ be an elliptic curve with a rational $n$-isogeny. Then
\[
n\leq 19 \text{ or } n \in\{21,25,27,37,43,67,163\}.
\]
\end{theorem}

Theorem~\ref{thm-isog} limits the primes $p$ for which $E(F)[p]$ can be cyclic.

\begin{lemma}\label{lem-finite_half_tor}
Let $E$ and $F$ be as in Theorem~\ref{thm-finite_over_F}. If $E(F)[p]$ has order $p$ then $p\leq 163$.
\end{lemma}
\begin{proof}
The group $H=E(F)[p]$ is stable under the action of $\Gal(F/\Q)$, hence Galois-stable.
If $|H|=p$, then $H$ is the kernel of a rational $p$-isogeny and $p\le 163$, by Theorem \ref{thm-isog}.
\end{proof}

Lemmas~\ref{lem-finite_full_tors} and~\ref{lem-finite_half_tor} together imply that for any elliptic curve $E/\Q$, the $p$-torsion subgroup of $E(F)$ is trivial for all but finitely many primes $p$, and $E[p^k]\subseteq E(F)$ for only finitely many prime powers $p^k$.
It remains only to check that the cyclic prime-power torsion of $E(F)$ is finite.

\begin{lemma}\label{lem-j-k_isog}
Let $E$ and $F$ be as in Theorem~\ref{thm-finite_over_F}, let $p$ be a prime, and let $k$ be the largest integer for which $E[p^k]\subseteq E(F)$.
If $E(F)_\tor$ contains a subgroup isomorphic to $\Z/p^k\Z\oplus\Z/p^j\Z$ with $j\ge k$, then~$E$ admits a rational $p^{j-k}$-isogeny. Moreover, $j-k\leq 4$, $3$, $2$, if $p=2$, $3$, $5$, respectively, $j-k\leq 1$ if $p=7,11,13,17,19,43,67,163$, and $j=k$ otherwise.
\end{lemma}

\begin{proof}
Let $Q\in E(F)$ be a point of order $p^j$, and choose $P\in E[p^j]$ so that $\{P,Q\}$ is a $\Z/p^j\Z$-basis for $E[p^j]$. If $\sigma\in\Gal(\Qbar/\Q)$, then $\sigma(Q)\in E(F)$, because $F$ is Galois, and $\sigma(Q)$ is a point of order~$p^j$.
Thus $\sigma(Q)\in E[p^j]$, so $\sigma(Q) = aP+bQ$ for some integers $a$ and $b$.

We claim that $a\equiv 0 \bmod p^{j-k}$. Indeed, the equality $\sigma(Q)=aP+bQ$ implies that
\[
aP = \sigma(Q)-bQ\in E(F),
\]
and if $t$ is the $p$-adic valuation of $a$, then $aP\in E[p^{j-t}]$ and $\{aP,p^tQ\}\subseteq E(F)$ is a $\Z/p^{j-t}\Z$-basis for $E[p^{j-t}]$. By the definition of $k$, we must have $j-t\leq k$, so $j-k\leq t$. Thus $a\equiv 0 \bmod p^{j-k}$, as claimed, and we may write $a=a'p^{j-k}$ for some integer $a'$.

Let $Q_{j-k}:=p^kQ\in E(F)$.
We claim that $\langle Q_{j-k}\rangle$ is $\Gal(\Qbar/\Q)$-stable. Indeed, we have
$$\sigma( Q_{j-k}) = \sigma(p^kQ) =p^k\sigma(Q) = p^k\Big(aP+bQ\Big) = p^k\Big(a'p^{j-k}P+bQ\Big) =a'p^jP+bp^kQ=bQ_{j-k},$$
for any $ \sigma\in\Gal(\Qbar/\Q)$.
Thus $\langle Q_{j-k} \rangle$ is a Galois-stable cyclic subgroup of $E(F)$ of order $p^{j-k}$, and $E\to E/\langle Q_{j-k} \rangle$ is a rational $p^{j-k}$-isogeny.   The bounds on $j-k$ then follow from Theorem~\ref{thm-isog}.
\end{proof}

\begin{proof}[Proof of Theorem \ref{thm-finite_over_F}] To show that $E(F)_\tor$ is finite, it suffices to show that (1) $E(F)_\tor$ has a non-trivial $p$-primary component for only finitely many primes $p$, and (2) for each of these primes~$p$, the $p$-primary component of $E(F)_\tor$ is finite.

\begin{enumerate}
\item Let $n$ be the maximum of $163$ and the largest order of a root of unity in $F$, and let $p>n$ be prime.
Then $E[p]\not\subseteq E(F)$, by Lemma \ref{lem-finite_full_tors}, so if the $p$-primary component of $E(F)_\tor$ is non-trivial, it must be cyclic, and in this case Lemma \ref{lem-finite_half_tor} implies that $p\le 163\le n$, which is a contradiction.
So the $p$-primary part of $E(F)_\tor$ is trivial for all $p>n$.

\item Let $p\le n$ be prime and let $k$ be the largest integer for which $\Q(\zeta_{p^k})\subseteq F$.
It follows from Lemma \ref{lem-j-k_isog} that the cardinality of the $p$-primary part of $E(F)_\tor$ is bounded by $p^{2k+4}$.
\end{enumerate}
The integer $n$ and the maximum value of $k$ over primes $p\le n$ depend only on $F$, as does the bound on $E(F)_\tor$.
This concludes the proof of Theorem \ref{thm-finite_over_F}.
\end{proof}

Lemma~\ref{lem-roots_of_unity} allows us to apply Theorem \ref{thm-finite_over_F} with $F=\Q(3^\infty)$; more generally, we have the following proposition.

\begin{prop}\label{prop-finite_over_Q(3)}
For every $d\ge 2$ the cardinality of $E(\Q(d^\infty))_\tor$ is finite and uniformly bounded as $E$ varies over elliptic curves over $\Q$.
\end{prop}
\begin{proof}
It follows from \cite[Prop.\ 10]{grizzard} that for any finite Galois extension $K/\Q$ in $\Q(d^\infty)$, the exponent of $\Gal(K/\Q)$ is bounded.
Indeed, $K$ is a subfield of a compositum of degree-$d$ fields, and $\Gal(K/\Q)$ is isomorphic to a quotient of a subgroup of a direct product of transitive groups of degree $d$, each of which has exponent dividing the exponent $\lambda(S_d)$ of the symmetric group $S_d$.
For all sufficiently large prime powers $p^k$, the exponent $\lambda(p^k) \ge p^k/4$ of $\Gal(\Q(\zeta_{p^k})/\Q)$ is larger than $\lambda(S_d)$, implying that $\zeta_{p^k}\not\in\Q(d^\infty)$.
The proposition then follows from Theorem~\ref{thm-finite_over_F}.
\end{proof}

We now make this result more precise in the case $d=3$ by determining the primes~$p$ for which $E(\Q(3^\infty))(p)$ can be non-trivial.  We first note the following lemma.

\begin{lemma}\label{lem-isog_gal}
Let $E/\Q$ be an elliptic curve that admits a rational $n$-isogeny $\varphi$, and let $P\in E[n]$ be a point of order $n$ in the kernel of $\varphi$.
The field extension $\Q(P)/\Q$ generated by the coordinates of~$P$ is Galois and $\Gal(\Q(P)/\Q)$ is isomorphic to a subgroup of $(\Z/n\Z)^\times$.
In particular, if $n$ is prime then $\Gal(\Q(P)/\Q)$ is cyclic and its order divides $n-1$.
\end{lemma}
\begin{proof}
The fact that $\varphi$ is rational implies that $\langle P\rangle$ is a Galois-stable subgroup of $E[n]$.
It follows that $\Q(P)/\Q$ is Galois: every Galois conjugate of a coordinate of $P$ is necessarily a coordinate of some $aP\in \langle P\rangle$, all of which lie in $\Q(P)$ because $E$ (and therefore the group law on $E$) is defined over~$\Q$.
The homomorphism $\Gal(\Q(P)/\Q)\rightarrow (\Z/n\Z)^\times$ given by $\sigma\mapsto a$, where $\sigma(P)=aP$, is injective, since if $\sigma(P)=\tau(P)$ then $\sigma\tau^{-1}(P)=P$, and this implies $\sigma\tau^{-1}=1$ fixes $\Q(P)$.
\end{proof}

\begin{proposition}\label{prop-possible_primes}
Let $E/\Q$ be an elliptic curve, and let $p$ be a prime dividing the cardinality of $E(\Q(3^\infty))_\tor$. Then $p\in \{2,3,5,7,13\}$.
\end{proposition}
\begin{proof}
For primes $p\ge 11$, Lemma \ref{lem-roots_of_unity} implies that $\Q(3^\infty)$ does not contain a primitive $p$th root of unity,  and therefore $E[p]\not\subseteq\Q(3^\infty)$, by Proposition~\ref{prop-weil}.
If $p>17$ with $p\neq 37, 43,67,163$, then Lemma \ref{lem-j-k_isog} implies that $E(\Q(3^\infty))[p]$ is trivial.

For the primes $p= 11,17,37, 43,67,$ and $163$, if the $p$-primary part $H$ of $E(\Q(3^\infty))_\tor$ is not trivial then it must be cyclic of order $p$, in which case $E$ admits a rational $p$-isogeny with a point $P\in E(\Q(3^\infty))[p]$ of order $p$ in its kernel.
By Lemma~\ref{lem-isog_gal}, the group $\Gal(\Q(P)/\Q)$ is cyclic, and it follows from Theorems 6.2 and 9.4 of \cite{lozano1} that its order is at least $(p-1)/2$ for $p\neq 37$, and at least $(p-1)/3=12$ for $p=37$.
In each case, the exponent of $\Gal(\Q(P)/\Q)$ cannot divide $6$, and therefore $\Q(P)\not\subseteq\Q(3^\infty)$, by Theorem~\ref{thm-galq3}.  But this contradicts $P\in E(\Q(3^\infty))$, so in fact $E(\Q(3^\infty))[p]$ must be trivial for all $p\ge 11$ except possibly $p=13$.  The proposition follows.
\end{proof}

As can be seen by the examples in Remark~\ref{rem-examples}, all the values of $p$ permitted by Proposition~\ref{prop-possible_primes} actually do arise for some $E/\Q$.
Lemmas~\ref{lem-roots_of_unity} and~\ref{lem-j-k_isog} imply explicit bounds on the prime powers $p^k$ that can divide $E(\Q(3^\infty))_\tor$ (namely, $k\le 10,7,2,3,1$ for $p=2,3,5,7,13$, respectively), but as we will show in the next section, these bounds are not tight.

\section{Maximal \texorpdfstring{$p$}{\textit{p}}-primary components of \texorpdfstring{$E(\Q(3^\infty))_\tor$}{\textit{E}(Q(3*))tors}}\label{sec-max}

In this section we obtain sharp bounds on the $p$-primary components of $E(\Q(3^\infty))$ for elliptic curves $E/\Q$. We will prove the following theorem.
\begin{theorem}\label{thm-upperbound}
Let $E/\Q$ be an elliptic curve. Then $E(\Q(3^\infty))_\tor$ is isomorphic to a subgroup of
\[
T_{\rm max}:= (\Z/8\Z\oplus\Z/16\Z)\oplus(\Z/9\Z\oplus\Z/9\Z)\oplus\Z/5\Z\oplus(\Z/7\Z\oplus\Z/7\Z)\oplus\Z/13\Z,
\]
and $T_{\rm max}$ is the smallest group with this property.
\end{theorem}

In order to prove the theorem it suffices to address the $p$-primary components $E(\Q(3^\infty))(p)$ for each of the primes $p=2,3,5,7,13$ permitted by Proposition~\ref{prop-possible_primes}.
We first prove two preliminary results that will be used in the subsections that follow.
We recall that the $\Qbar$-isomorphism class of an elliptic curve $E/\Q$ may be identified with its $j$-invariant $j(E)$.

\begin{proposition}\label{prop-twist}
Let $E/\Q$ be an elliptic curve with $j(E)\ne 1728$.
The isomorphism type of $E(\Q(3^\infty))_\tor$ depends only on the $\Qbar$-isomorphism class of $E$, equivalently, only on $j(E)$.
\end{proposition}
\begin{proof}
Recall that for $j(E)\ne 0,1728$, if $j(E')=j(E)$ for some $E'/\Q$ then $E'$ is isomorphic to $E$ over $\Q(\sqrt{a})$, for some $a\in\Q$,
and if $j(E)=0$ and $j(E')=j(E)$, then $E'/\Q$ is isomorphic to $E$ over $\Q(\sqrt[6]{a})$, for some $a\in\Q$; see \S X.5 of \cite{silverman}.
In either case, the Galois closure of the minimal extension of $\Q$ over which the base changes of $E$ and $E'$ become isomorphic is isomorphic to a subgroup of $S_3\times C_2$, hence of generalized $S_3$-type.  It follows that if $j(E)=j(E')\ne 1728$ for some $E'/\Q$ then the base changes of $E$ and $E'$ to $\Q(3^\infty)$ are isomorphic and therefore $E(\Q(3^\infty))_\tor\simeq E'(\Q(3^\infty))_\tor$.
\end{proof}

\begin{remark}\label{rem-twist}
When $j(E)=1728$ there are two possibilities: either $E(\Q(3^\infty))_\tor\simeq\Z/2\Z\oplus\Z/2\Z$ or $E(\Q(3^\infty))_\tor\simeq\Z/4\Z\oplus\Z/4\Z$.  These are realized by the elliptic curves \href{http://www.lmfdb.org/EllipticCurve/Q/256b1}{\texttt{256b1}} and \href{http://www.lmfdb.org/EllipticCurve/Q/32a1}{\texttt{32a1}}, respectively. \end{remark}

\begin{lemma}\label{lem-divisibility}
Let $p$ and $q$ be distinct primes, let $K_2/K_1$ be a finite Galois extension of number fields with $[K_2:K_1]$ a power of $q$, and let $E$ be an elliptic curve defined over $\Q$.
\begin{enumerate}
\item If $E(K_1)[p]=E(K_2)[p]$, then $E(K_1)(p)=E(K_2)(p)$, that is, if the $p$-torsion of $E$ does not grow when we move from $K_1$ to $K_2$, then neither does the $p$-primary torsion.
\item Let $\mathcal{P}=E(K_2)[p]$.  Then $E(K_1(\mathcal{P}))(p)= E(K_2)(p)$, that is, the $p$-primary torsion of $E(K_2)$ stabilizes over the extension of $K_1$ generated by the the $p$-torsion of $E(K_2)$.
\end{enumerate}
\end{lemma}

\begin{proof}
We first note that (2) is implied by (1), since $K_1(\mathcal{P})$ has all the properties required of $K_1$ (indeed, $K_1\subseteq K_1(\mathcal{P})\subseteq K_2$, so $K_2/K_1(\mathcal{P})$ and $[K_2:K_1(\mathcal{P})]$ divides $[K_2:K_1]$, so it is a power of $q$).

To prove (1), we assume $E(K_1)[p]=E(K_2)[p]$; (1) clearly holds when this group is trivial, we assume otherwise.
We now suppose for the sake of contradiction that $E(K_1)(p)$ is properly contained in $E(K_2)(p)$.
Then there exists a point $Q\in E(K_2)(p)$ for which $P=pQ$ is a nonzero point in $E(K_1)(p)$, say of order~$p^k$ for some $k\ge 1$.  Then $R=p^{k-1}P$ is a nonzero element of $E(K_1)[p]\subseteq E[p]$, and we may choose $S\in E[p]$ so that $\{R,S\}$ is a $\Z/p\Z$-basis for $E[p]$.

The multiplication-by-$p$ map is a separable endomorphism of degree $p^2$, so there are $p^2$ distinct preimages of~$P$ under multiplication by $p$ (including $Q$); these are precisely the points in the set
\[
\mathcal{Q}:=[p]^{-1}(P)=\{Q+aR + bS : 0\leq a,b < p\}.
\]
Put $\mathcal{Q}_1:=\mathcal{Q}\cap E(K_1)$ and $\mathcal{Q}_2:=\mathcal{Q}\cap E(K_2)$.
Of the $p^2$ points in $\mathcal{Q}$, at least $p$ lie in $E(K_2)$, namely, the points $Q+aR$ (since $Q,R\in E(K_2)$), so $\mathcal{Q}_2$ has cardinality at least $p$.
If its cardinality is greater than $p$, then  $Q+aR+bS\in E(K_2)$ for some $b\not\equiv 0\bmod p$, which implies $bS\in E(K_2)$, and therefore $S\in E(K_2)$, since $b$ is invertible modulo $p$ and $S$ has order $p$.
Thus the cardinality of $\mathcal{Q}_2$ is either $p^2$ or $p$, depending on whether $E(K_2)[p]=E[p]$ or not.

We claim that $\mathcal{Q}_1$ is empty.  For the sake of contradiction, suppose $Q+aR+bS\in \mathcal{Q}_1\subseteq E(K_1)$. We then have $Q+bS\in E(K_1)$, since $R\in E(K_1)$, and since $Q\notin E(K_1)$ by assumption, we must have $b\not\equiv 0 \bmod p$.  This implies $S\in E(K_2)$, since $Q,Q+bS\in E(K_2)$.
But then $S\in E(K_2)[p]=E(K_1)[p]$, so $S\in E(K_1)$, which contradicts $Q\notin E(K_1)$, since $Q+bS\in E(K_1)$.

The Galois group $\Gal(K_2/K_1)$ acts on the set $\mathcal{Q}$, since it is the solution set of $pX=P$, which is stable under $\Gal(K_2/K_1)$ because $P\in E(K_1)$.
The fact that $\mathcal{Q}_1$ is empty implies that this action has no fixed points.
By the orbit-stabilizer theorem, the size of each orbit divides $|\Gal(K_2/K_1)|$, a power of the prime $q$, and since no orbit is trivial, the size of each orbit is divisible by~$q$. It follows that the cardinality $p^2$ of $\mathcal{Q}$ is divisible by $q$, which is a contradiction, since $p$ and $q$ are distinct primes.
Thus our supposition that $E(K_1)(p)\ne E(K_2)(p)$ must be false, which proves (1).
\end{proof}

\subsection{Primes without the possibility of full \texorpdfstring{$\boldsymbol{p}$}{\textit{p}}-torsion (\texorpdfstring{$\boldsymbol{p=5,13}$}{\textit{p}=5,13})}
We start with the primes~$p$ for which $E[p]\not\subseteq E(\Q(3^\infty))$, namely, $p=5,13$.
In these cases $E(\Q(3^\infty))(p)$ is necessarily cyclic.

\begin{lemma}\label{lem-5tor}
Let $E/\Q$ be an elliptic curve.
Then $E(\Q(3^\infty))(5)$ is either trivial or isomorphic to $\Z/5\Z$; the latter holds if and only if $E$ admits a rational $5$-isogeny whose kernel generates an extension of degree at most $2$.
\end{lemma}

\begin{proof}
It follows from Lemma~\ref{lem-roots_of_unity} and Proposition~\ref{prop-weil} that $E[5]\not\subseteq E(\Q(3^\infty))$, thus $E(\Q(3^\infty))(5)$ is cyclic of order $5^j$ for some $j\ge 0$.
Lemma \ref{lem-j-k_isog} implies, $j\le 2$; we will show that in fact $j\le 1$.
Suppose for the sake of contradiction that $E(\Q(3^\infty))$  contains a point $P$ of order 25.
Let $K:=\Q(P)\subseteq \Q(3^\infty)$, let $K_2\subseteq \Q(3^\infty)$ be the Galois closure of~$K$, and let $K_1:=K_2\cap \Q(2^\infty)$.
Then $[K_2:K_1]$ is a power of~$3$, since  $\Gal(K_2/\Q)$ is of generalized $S_3$-type and
$\Q(3^\infty)/\Q(2^\infty)$ is elementary 3-abelian.
Theorem~\ref{thm-fujita} then implies that $E(K_1)(5)\subseteq E(\Q(2^\infty))(5)$ is either trivial or isomorphic to $\Z/5\Z$.

Suppose first that $E(K_1)(5)$ is trivial. The point $P_1=5P\in E(K_2)$ has order $5$, but $E[5]\not\subseteq E(K_2)$, since $K_2\subseteq E(\Q(3^\infty))$, so $\langle P_1 \rangle\subseteq E(K_2)$ is Galois-stable and therefore the kernel of a rational 5-isogeny.
This implies that $G:=\Gal(\Q(P_1)/\Q)$ is isomorphic to a subgroup of $(\Z/5\Z)^\times$, by Lemma~\ref{lem-isog_gal}.
The group $G$ cannot have order $4$, because it is the Galois group of a number field in $\Q(3^\infty)$ and must have exponent dividing $6$, by Theorem \ref{thm-galq3}.
On the other hand, $G$ cannot have order $1$ or $2$, because then $P_1$ would be defined over a quadratic extension, and therefore over $K_1=K_2\cap \Q(2^\infty)$, contradicting our assumption that $E(K_1)(5)$ is trivial.

We therefore must have $E(K_1)(5)\simeq \Z/5\Z$, in which case $E(K_1)[5]=E(K_2)[5]\simeq \Z/5\Z$, and we may apply Lemma \ref{lem-divisibility} with $p=5$ and $q=3$.
But then $E(K_1)(5)=E(K_2)(5)$, which contradicts our assumption that $E(K_2)$ contains a point of order $25$.
So $j\le 1$ as claimed and $E(\Q(3^\infty))(5)$ is either trivial or isomorphic to $\Z/5\Z$.
In the latter case $E(\Q(3^\infty))(5)$ is a Galois-stable cyclic subgroup of order 5 that is the kernel of a rational 5-isogeny.
It follows from Lemma~\ref{lem-isog_gal} that this kernel generates a cyclic extension $K/\Q$ whose degree divides 4, and in fact it must have degree 2, since $K\subseteq\Q(3^\infty)$ implies that the exponent of $\Gal(K/\Q)$ divides 6.
Conversely, if $E$ admits a rational 5-isogeny whose kernel generates an extension $K/\Q$ of degree at most $2$, then  $K\subseteq\Q(3^\infty)$, by Theorem~\ref{thm-galq3conv}, in which case $E(\Q(3^\infty))(5)\simeq\Z/5\Z$.
\end{proof}

\begin{example}\label{ex-5max}
Any elliptic curve $E/\Q$ with a rational point of order 5 has $E(\Q(3^\infty))(5)\simeq \Z/5\Z$; the curve \href{http://www.lmfdb.org/EllipticCurve/Q/11a1}{\texttt{11a1}} is an example.
Another example is the curve \href{http://www.lmfdb.org/EllipticCurve/Q/50a3}{\texttt{50a3}}, which has trivial rational 5-torsion but admits a rational 5-isogeny whose kernel generates an extension of degree 2.
\end{example}

\begin{lemma}\label{lem-13tor}
Let $E/\Q$ be an elliptic curve. Then, $E(\Q(3^\infty))(13)$ is either trivial or isomorphic to $\Z/13\Z$; the latter holds if and only if $E$ admits a rational $13$-isogeny whose kernel generates an extension of degree dividing $6$.
\end{lemma}
\begin{proof}
It follows from Lemma~\ref{lem-roots_of_unity} and Proposition~\ref{prop-weil} that $E[13]\not\subseteq E(\Q(3^\infty))$, thus $E(\Q(3^\infty)))$ is cyclic of order $13^j$ for some $j\ge 0$, and Lemma \ref{lem-j-k_isog} implies $j\le 1$.
The last statement follows from Lemma \ref{lem-isog_gal}: the kernel of a 13-isogeny admitted by $E$ generates a cyclic extension $K/\Q$ of degree dividing 12, and then $K\subseteq\Q(3^\infty)$ if and only $[K:\Q]$ divides 6, by Theorems ~\ref{thm-galq3} and \ref{thm-galq3conv}.
\end{proof}

\begin{example}\label{ex-13max}
The curve \href{http://www.lmfdb.org/EllipticCurve/Q/147b1}{\texttt{147b1}} has $E(\Q(3^\infty))\simeq \Z/13\Z$; its 13-division polynomial has a cubic factor, so it has a point of order 13 over an extension whose degree divides 6 (in fact, 3).
\end{example}

\subsection{Primes with the possibility of full \texorpdfstring{$\boldsymbol{p}$}{\textit{p}}-torsion (\texorpdfstring{$\boldsymbol{p=2,3,7}$}{\textit{p}=2,3,7})}
We now consider the primes $p=2,3,7$ for which $\Q(3^\infty)$ contains a primitive $p$th root of unity (so $E[p]\subseteq E(\Q(3^\infty))$ is not immediately ruled out by the Weil pairing).
In this subsection we address $p=2,7$; the case $p=3$ is addressed in the next subsection.

\begin{lemma}\label{lem-2tor}
If $E/\Q$ is an elliptic curve, then $E(\Q(3^\infty))[2]=E[2]\simeq\Z/2\Z\oplus\Z/2\Z.$
\end{lemma}
\begin{proof}
If we put $E/\Q$ in the form $y^2=f(x)$ with $f(x)$ cubic, the non-trivial points in $E[2]$ are precisely the points of the form $(\alpha,0)$ with $\alpha$ a root of $f$, all of which lie in $\Q(3^\infty)$.
\end{proof}

\begin{lemma}\label{lem-2primary}
Let $E/\Q$ be an elliptic curve.
If $E(\Q)[2]$ is non-trivial then $E(\Q(3^\infty))(2)$ is equal to $E(\Q(2^\infty))(2)$; otherwise $E(\Q(3^\infty))(2)$ is equal to $E[2]$ or $E[4]$ and $E(\Q(2^\infty))(2)$ is trivial.
In either case, $E(\Q(3^\infty))$ is isomorphic to a subgroup of $\Z/8\Z\oplus \Z/8\Z$ or $\Z/4\Z\oplus \Z/16\Z$.
\end{lemma}

\begin{proof}
We first suppose that $E(\Q)[2]$ is non-trivial.
Then $E(\Q(2^\infty))[2]$ is also non-trivial, and therefore $E(\Q(2^\infty))[2] = E[2]$, by Theorem~\ref{thm-fujita}.
Lemma \ref{lem-divisibility} then implies that the 2-primary torsion cannot grow in any 3-power Galois extension of $\Q(E[2])$. Since $\Q(E[2])\subseteq \Q(2^\infty)\subseteq \Q(3^\infty)$, we must have $E(\Q(3^\infty))(2)=E(\Q(2^\infty))(2)$, and Theorem~\ref{thm-fujita} then implies that $E(\Q(3^\infty))(2)$ is isomorphic to a subgroup of $\Z/8\Z\oplus\Z/8\Z$ or $\Z/4\Z\oplus\Z/16\Z$.

We now suppose that $E(\Q)[2]$ is trivial.
Then $E(\Q(2^\infty))(2)$ is also trivial: if $E\colon y^2=f(x)$ has no rational points of order 2 then the cubic $f$ must be irreducible, in which case every point of order 2 generates a field of degree 3.
We also note that $E$ cannot admit a rational 2-isogeny, since the unique point of order 2 in the kernel of such an isogeny would be Galois-stable, hence rational.
Thus~$E$ does not admit a rational $2^j$-isogeny for any $j>0$;
Lemma \ref{lem-j-k_isog} then implies $E(\Q(3^\infty))(2)\simeq \Z/2^k\Z \times \Z/2^k\Z$ for some $k\ge 0$, and
Proposition~\ref{prop-weil} and Lemma~\ref{lem-roots_of_unity} imply $k\le 3$.
To show $k<3$, we note that an enumeration (in Magma) of the subgroups $G$ of $\GL_2(\Z/8\Z)$ with surjective determinant maps finds that whenever $G$ is of generalized $S_3$-type, it is actually elementary 2-abelian.
This implies that if $\Q(E[8])\subseteq \Q(3^\infty)$ then in fact $\Q(E[8])\subseteq\Q(2^\infty)$, but we have assumed that $E(\Q[2])$ is trivial, hence $E(\Q(2^\infty))(2)$ is trivial, so this cannot occur.
\end{proof}

\begin{example}\label{ex-2max}
The elliptic curves \href{http://www.lmfdb.org/EllipticCurve/Q/15a1}{\texttt{15a1}} and  \href{http://www.lmfdb.org/EllipticCurve/Q/210e2}{\texttt{210e2}} realize the maximal possibilities $\Z/8\Z\oplus\Z/8\Z$ and $\Z/4\Z\oplus\Z/16\Z$, respectively, for $E(\Q(3^\infty))(2)$.
\end{example}

Before addressing the $7$-primary component of $E(\Q(3^\infty))$, we prove a lemma that relates the degree of the $p$-torsion field $\Q(E[p])$ of $E/\Q$ to the number of rational $p$-isogenies admitted by $E$ (we consider two isogenies to be distinct only if their kernels are distinct).

\begin{lemma}\label{lem-piso}
Let $E/\Q$ be an elliptic curve and let $p>2$ be a prime for which $E$ admits a rational $p$-isogeny. Then $[\Q(E[p]):\Q]$ is relatively prime to $p$ if and only if $E$ admits two rational $p$-isogenies (with distinct kernels).
For $p>5$ this implies that $p$ divides $[\Q(E[p]):\Q]$.
\end{lemma}
\begin{proof}
The hypothesis implies that the image of $\rho_{E,p}$ is conjugate to a subgroup $B$ of the Borel group in $\GL_2(\Z/p\Z)$.  Lemma~2.2 of \cite{lozano4} implies that $B=B_dB_1$ where
\[
B_1 := B\cap \left\{ \left(\begin{array}{cc} 1 & b \\ 0 & 1\end{array}\right) : b\in\Z/p\Z \right\}, \text{ and }
B_d := B\cap \left\{ \left(\begin{array}{cc} a & 0 \\ 0 & c\end{array}\right) : a,c\in(\Z/p\Z)^\times \right\}.
\]
Thus the order of $B\simeq \Gal(\Q(E[p])/\Q)$ is relatively prime to $p$ if and only if $B_1$ is trivial, equivalently, $B=B_d$ is a subgroup of the split Cartan group, in which case $E$ admits two rational $p$-isogenies with distinct kernels.  However, as proved in \cite{kenku}, this can only occur for $p\leq 5$.
\end{proof}

\begin{lemma}\label{lem-7tor}
Let $E/\Q$ be an elliptic curve.
Then $E(\Q(3^\infty))(7)$ is isomorphic to a subgroup of $\Z/7\Z\oplus\Z/7\Z$.
The case $E(\Q(3^\infty))(7)\simeq \Z/7\Z\oplus\Z/7\Z$ occurs if and only if $j(E)=2268945/128$, and the case $E(\Q(3^\infty))(7)\simeq \Z/7\Z$ occurs if and only if $E$ admits a rational $7$-isogeny, equivalently,
\[
j(E) = \frac{(t^2+13t+49)(t^2+5t+1)^3}{t},
\]
for some $t\in\Q^\times$.
\end{lemma}

\begin{proof}
Lemma \ref{lem-roots_of_unity} and Proposition \ref{prop-weil} imply that $E[49]\not\subseteq \Q(3^\infty)$,
and Lemma \ref{lem-j-k_isog} then implies that $E(\Q(3^\infty))(7)\simeq \Z/7^k\Z\oplus\Z/7^j\Z$ with $k\le 1$, and $k\le j\le k+1$.

If $j>k$ then Lemma \ref{lem-j-k_isog} implies that $E$ admits a rational 7-isogeny, and Lemma \ref{lem-piso} then implies that $[\Q(E[7]):\Q]$ is divisible by 7.
The exponent of $\Gal(\Q(E[7])/\Q)$ is therefore not divisible by~6, so $\Q(E[7])\not\subseteq\Q(3^\infty)$, therefore $k=0$, $j=1$, and $E(\Q(3^\infty))(7)\simeq \Z/7\Z$.
This also rules out the case $k=1$ and $j=2$, which proves the first statement in the theorem.

If $j=k$ then we claim that $E$ cannot admit a rational 7-isogeny.
Indeed, if $E$ admits a rational 7-isogeny and $P$ is a non-trivial point in its kernel, then Lemma~\ref{lem-isog_gal} implies that $\Gal(\Q(P)/\Q)$ is cyclic of order dividing 6, hence of generalized $S_3$-type, so $\Q(P)\in \Q(3^\infty)$, by Theorem~\ref{thm-galq3conv}.
But then we must have $j=k=1$, so $\Q(E[7])\subseteq \Q(3^\infty)$, but then Lemma \ref{lem-piso} implies that 7 divides $[\Q(E[7]):\Q]$, which contradicts $\Q(E[7])\subseteq \Q(3^\infty)$.
Thus $k=0$,$j=1$ if and only if $E$ admits a rational 7-isogeny, equivalently, $j(E)$ lies in the image of the map from $X_0(7)$ to the $j$-line that appears in the statement of the lemma and can be found in \cite[Table 3]{lozano1}, for example.

If $j=k=1$ then $\Q(E[p])\subseteq \Q(3^\infty)$, so $\Gal(\Q(E[p])/\Q)$ has exponent dividing~6, by Theorem~\ref{thm-galq3}.
This implies that for every prime $p\ne 7$ of good reduction for $E$, the elliptic curve $E_p/\F_p$ obtained by reducing $E$ modulo $p$ has its 7-torsion defined over an $\F_p$-extension of degree dividing~6, and in particular, admits an $\F_p$-rational 7-isogeny (two in fact).  Thus $E/\Q$ admits a rational 7-isogeny locally everywhere but not globally, and as proved in \cite{sutherland1}, this implies $j(E)=2268945/128$.
Conversely, as also proved in \cite{sutherland1}, for every elliptic curve $E/\Q$ with this $j$-invariant the group $\Gal(\Q(E[7])/\Q)$ is isomorphic to a subgroup of $\GL_2(\F_7)$ with surjective determinant map whose image in $\PGL_2(\F_7)$ is isomorphic to $S_3$; up to conjugacy there are exactly two such groups (labeled \texttt{7NS.2.1} and \texttt{7NS.3.1} in \cite{sutherland2}), and both are of generalized $S_3$-type.
Thus every elliptic curve $E/\Q$ with $j(E)=2268945/128$ has $E(\Q(3^\infty))(7)\simeq \Z/7\Z\oplus \Z/7\Z$.

Otherwise, $j=k=0$ and $E(\Q(3^\infty))(7)$ is trivial; the lemma follows.
\end{proof}

\begin{example}\label{ex-7max}
The curve \href{http://www.lmfdb.org/EllipticCurve/Q/2450a1}{\texttt{2450a1}} has $j$-invariant $2268945/128$ and is thus an example of an elliptic curve $E/\Q$ for which $E(\Q(3^\infty))(7)\simeq\Z/7\Z\oplus\Z/7\Z$.
\end{example}

\begin{corollary}\label{cor-14x14}
Let $E/\Q$ be an elliptic curve. Then $E(\Q(3^\infty))_\tor\simeq \Z/14\Z\oplus\Z/14\Z$ if and only if $j(E)=2268945/128$.
\end{corollary}
\begin{proof}
The forward implication is an immediate consequence of Lemmas~\ref{lem-2tor} and~\ref{lem-7tor}.
A direct computation of $E(\Q(3^\infty))(p)$ for $p=2,3,5,7,13$ for the elliptic curve  \href{http://www.lmfdb.org/EllipticCurve/Q/2450a1}{\texttt{2450a1}} in Example~\ref{ex-7max} finds that $E(\Q(3^\infty))_\tor = E[14]$ for this particular $E/\Q$ with $j(E)=2268945/128$, hence for every $E/\Q$ with the same $j$-invariant, by Proposition~\ref{prop-twist}.
\end{proof}

\subsection{The 3-primary component of \texorpdfstring{$\boldsymbol{E(\Q(3^\infty))_\tor}$}{\textit{E}(Q(3*))tors}}

\begin{lemma}\label{lem-3tor}
Let $E/\Q$ be an elliptic curve.
Then $E(\Q(3^\infty))[3]=E[3]$ if and only if $E$ admits a rational $3$-isogeny, and $E(\Q(3^\infty))(3)$ is trivial otherwise.
\end{lemma}
\begin{proof}
An enumeration of the subgroups $G$ of $\GL_2(\Z/3\Z)$ finds that $G$ is of generalized $S_3$-type if and only if it is conjugate to a subgroup of the Borel group; this implies the first part of the lemma, since $E(\Q(3^\infty))[3]=E[3]$ if and only if $\Gal(\Q(E[3])/\Q)\simeq \im \rho_{E,3}\subseteq \GL_2(\Z/3\Z)$ is of generalized $S_3$-type.
If $\Q(E[3])\not\subseteq \Q(3^\infty)$, then Lemma~\ref{lem-j-k_isog} implies that if $E(\Q(3^\infty))(3)$ is non-trivial then $E$ admits a rational 3-isogeny, but this cannot occur, so $E(\Q(3^\infty))(3)$ is trivial.
\end{proof}

\begin{lemma}\label{lem-9x27}
Let $E/\Q$ be an elliptic curve. Then $E(\Q(3^\infty))$ does not contain a subgroup isomorphic to $\Z/9\Z\oplus  \Z/27\Z.$
\end{lemma}
\begin{proof}
Suppose for the sake of contradiction that there is an elliptic curve $E/\Q$ for which $E(\Q(3^\infty))$ contains a subgroup isomorphic to $\Z/9\Z\oplus\Z/27\Z$.
Then the image $G:=\im\rho_{E,27}\subseteq \GL_2(\Z/27\Z)$ of the mod-27 Galois representation attached to $E$ satisfies the following properties:
\begin{enumerate}[(i)]
\item $G$ has a surjective determinant map and an element with trace 0 and determinant $-1$;
\item $G$ contains a normal subgroup $N$ that acts trivially on a $\Z/27\Z$-submodule of $\Z/27\oplus\Z/27$ isomorphic to $\Z/9\Z\oplus\Z/27\Z$ for which $G/N$ is of generalized $S_3$-type.
\end{enumerate}
As noted in \S\ref{sec-notation}, the first condition is required by $\rho_{E,n}$ for any elliptic curve $E/\Q$.
The second requirement reflects the fact that $\Q(E[27])$ contains the Galois extension $\Q(E(\Q(3^\infty))[27])/\Q$ whose Galois group is a quotient $G/N$ of $G$ and for which the Galois group $\Gal(\Q(E[27]/\Q(E(\Q(3^\infty))[27])\simeq N$ acts trivially on a subgroup of $E[27]$ isomorphic to $\Z/9\Z\oplus\Z/27\Z$.

An enumeration in Magma of the subgroups of $\GL_2(\Z/27\Z)$ finds that every such $G$ is conjugate to a subgroup of the full inverse image of
\[
H:=\left\langle
\begin{pmatrix}1&3\\0&1\end{pmatrix},
\begin{pmatrix}1&0\\0&2\end{pmatrix},
\begin{pmatrix}8&0\\0&1\end{pmatrix}
\right\rangle\subseteq\GL_2(\Z/9\Z)
\]
in $\GL_2(\Z/27\Z)$.
Taking the intersection of $H$ with $\SL_2(\Z/9\Z)$ shows that $H$ corresponds to the congruence subgroup labeled \href{http://www.uncg.edu/mat/faculty/pauli/congruence/csg1.html#group9H1}{$\textrm{9H}^1$} in the tables of Cummins and Pauli \cite{cp}.
The modular curve $X_H$ of level 9 and genus 1 is defined over $\Q$ and has 3 rational cusps (the number of rational cusps can be determined via \cite[Lemma 3.4]{zywina2}, for example).
The group $H$ is equal to the intersection $H_1\cap H_2$ of two subgroups of $\GL_2(\Z/9\Z)$ whose intersection with $\SL_2(\Z/9\Z)$ gives the congruence subgroups \href{http://www.uncg.edu/mat/faculty/pauli/congruence/csg0.html#group9I0}{$\textrm{9I}^0$} and \href{http://www.uncg.edu/mat/faculty/pauli/congruence/csg0.html#group9J0}{$\textrm{9J}^0$}.
Explicit rational parameterizations for the genus zero modular curves $X_{H_1}$ and $X_{H_2}$ appear in \cite{SZ}; these curves both admit rational maps to $X_0(3)$, allowing us to explicitly construct a rational model for $X_H$ as the fiber product of these maps over $X_0(3)$.  This model is isomorphic to the elliptic curve \href{http://www.lmfdb.org/EllipticCurve/Q/27a3}{\texttt{27a3}}, which has just 3 rational points, which is equal to the number of rational cusps on $X_H$, so there are no non-cuspidal rational points.
It follows that for every elliptic curve $E/\Q$, the image of $\rho_{E,27}$ is not conjugate to a subgroup of $H$, which is our desired contradiction.
\end{proof}

\begin{proposition}\label{prop-no27}
If $E/\Q$ is an elliptic curve, then $E(\Q(3^\infty))$ does not have a point or order $27$.
\end{proposition}

\begin{proof}
Suppose for the sake of obtaining a contradiction that $E/\Q$ is an elliptic curve with a point of order 27 defined over $\Q(3^\infty)$.
Lemmas \ref{lem-3tor} and \ref{lem-9x27} imply $E(\Q(3^\infty))(3)\simeq \Z/3\Z\times\Z/27\Z$.
We now proceed as in the proof of Lemma~\ref{lem-9x27}, and consider the subgroups $G$ of $\GL_2(\Z/27\Z)$ that may arise as the image of the mod-27 Galois image $\im\rho_{E,27}$, except in (ii) we now only require the normal subgroup $N$ of $G$ for which $G/N$ is of generalized $S_3$-type to act trivially on a submodule isomorphic to $\Z/3\Z\oplus\Z/27\Z$.
We find that every such $G$ is conjugate to a subgroup of one of three subgroups $H_1,H_2,H_3\subseteq \GL_2(\Z/27\Z)$ whose intersection with $\SL_2(\Z/27\Z)$ yields the congruence subgroups with Cummins-Pauli labels \href{http://www.uncg.edu/mat/faculty/pauli/congruence/csg1.html#group27C1}{$\textrm{27C}^1$}, \href{http://www.uncg.edu/mat/faculty/pauli/congruence/csg4.html#group27B4}{$\textrm{27B}^4$}, \href{http://www.uncg.edu/mat/faculty/pauli/congruence/csg4.html#group27A4}{$\textrm{27A}^4$}, respectively.
We now show that no elliptic curve $E/\Q$ can have $\im\rho_{E,27}$ conjugate to a subgroup of any of the groups $H_1,H_2,H_3$, which is our desired contradiction.

The group $H_1$ lies in the Borel subgroup of upper triangular matrices in $\GL_2(\Z/27\Z)$, so if $\im\rho_{E,27}$ is conjugate to a subgroup of $H_1$ then $E$ admits a rational 27-isogeny.
From \cite[Table~4]{lozano1} we see that there is just one $\Qbar$-isomorphism class of elliptic curves that admit a rational 27-isogeny, represented by the curve \href{http://www.lmfdb.org/EllipticCurve/Q/27a2}{\texttt{27a2}}.
None of the four curves in its isogeny class \href{http://www.lmfdb.org/EllipticCurve/Q/27a}{\texttt{27a}} have $j$-invariant 1728, so by Proposition~\ref{prop-twist}, it is enough to check whether $E(\Q(3^\infty))$ contains a point of order~27 for each of the four curves $E/\Q$ in isogeny class \href{http://www.lmfdb.org/EllipticCurve/Q/27a}{\texttt{27a}}; a direct computation finds that none do.

The group $H_2$ is conjugate to a subgroup of
\[
H_4:=\left\langle
\begin{pmatrix}1&1\\9&1\end{pmatrix},
\begin{pmatrix}1&0\\0&2\end{pmatrix}
\begin{pmatrix}2&0\\0&1\end{pmatrix}
\right\rangle\subseteq\GL_2(\Z/27\Z),
\]
whose intersection with $\SL(2,\Z/27\Z)$ is conjugate to \href{http://www.uncg.edu/mat/faculty/pauli/congruence/csg2.html#group27A2}{$\textrm{27A}^2$}.
Using the methods of \cite{RZB}, Rouse and Zureick-Brown have computed a model for the corresponding modular curve $X_{H_4}$ of genus 2, which has two rational cusps:
\[
X_{H_4}: y^2 = x^6 - 18x^3 - 27.
\]
A 2-descent on the Jacobian of this curve shows that it has rank zero, so the rational points on $X_{H_4}$ can be easily determined via Chabauty's method (using the \texttt{Chabauty0} function in Magma, for example).
The only points in $X_{H_4}(\Q)$ are the 2 points at infinity, both of which must be cusps.
This rules out the possibility that $\im\rho_{E,27}$ is conjugate to a subgroup of $H_2\subseteq H_4$.

This leaves only the group
\[
H_3:=\left\langle
\begin{pmatrix}1&2\\9&1\end{pmatrix},
\begin{pmatrix}1&0\\0&2\end{pmatrix},
\begin{pmatrix}8&0\\0&1\end{pmatrix}
\right\rangle \subseteq \GL_2(\Z/27\Z).
\]
Using the results of \cite{SZ}, a singular model for the modular curve $X_{H_3}$ can be explicitly constructed as the fiber product over $X_0(9)$ of two genus zero curves with maps $t^3$ and $ (t^3-6t^2+3t+1)/(t^2-t)$ to $X_0(9)$ (the corresponding congruence subgroups are \href{http://www.uncg.edu/mat/faculty/pauli/congruence/csg0.html#group27A0}{$\textrm{27A}^0$} and \href{http://www.uncg.edu/mat/faculty/pauli/congruence/csg0.html#group9I0}{$\textrm{9I}^0$}, respectively).
This yields the genus 4 curve
\[
X_{H_3}\colon x^3y^2 - x^3y - y^3 + 6y^2 - 3y = 1.
\]
which has two rational points at infinity (both singular).

Over $\Q(\zeta_3)$ the automorphism group of $X_{H_3}$ is isomorphic to $\Z/3\Z\oplus\Z/3\Z$, and with a suitable choice of basis for $\Aut(X_{H_3})$ the two cyclic factors yield two distinct genus 2 quotients, corresponding to the curve
\[
C\colon y^2 = x^6-18\zeta_3x^3-27\zeta_3^2
\]
and its complex conjugate $\overline{C}$. The curve $C$ is isomorphic to $X_{H_4}$ over $\Q(\zeta_9)$, consistent with the fact that the restriction of $H_3$ to elements with determinant $1\bmod 9$ is a subgroup of $H_4$.
A calculation by Jackson Morrow (see \cite{magmascripts} for details) shows that the Jacobian of $C$ has rank 0 and torsion subgroup of order 3 generated by the difference of the two points at infinity on $C$ (and similarly for $\overline{C}$).
It follows that the only rational points on $C$ and $\overline{C}$ are the points at infinity; pulling back these points to our model for $X_{H_3}$ yields only the two rational points at infinity, both of which correspond to cusps on $X_{H_3}$; this rules out the possibility that $\im\rho_{E,27}$ is conjugate to a subgroup of $H_3$.
\end{proof}

Having ruled out points of order 27 in $E(\Q(3^\infty))_\tor$, we now give a necessary and sufficient criterion for $E(\Q(3^\infty))(3)$ to be maximal.

\begin{lemma}\label{lem-9x9}
Let $E/\Q$ be an elliptic curve.
Then $E(\Q(3^\infty))(3)=E[9]\simeq \Z/9\Z\oplus\Z/9\Z$ if and only if one of the following holds:
\begin{itemize}
\item[(i)] The image of $\rho_{E,3}$ is conjugate to a subgroup of the split Cartan subgroup of $\GL_2(\Z/3\Z)$; equivalently, $E$ admits two distinct rational $3$-isogenies.  This case occurs if and only if
\[
j(E) = \frac{27t^3(8-t^3)^3}{(t^3+1)^3},
\]
for some $t\in \Q$, $t \ne -1$.
\item[(ii)] The image of $\rho_{E,9}$ is conjugate in $\GL_2(\Z/9\Z)$ to a subgroup of
\[
H:=\left\langle
\begin{pmatrix}1&2\\3&1\end{pmatrix},
\begin{pmatrix}1&3\\0&1\end{pmatrix},
\begin{pmatrix}1&0\\0&8\end{pmatrix},
\begin{pmatrix}2&0\\0&2\end{pmatrix}
\right\rangle.
\]
This case occurs if and only if
\[
j(E) = \frac{432t(t^2-9)(t^2+3)^3(t^3-9t+12)^3(t^3+9t^2+27t+3)^3(5t^3-9t^2-9t-3)^3}{(t^3-3t^2-9t+3)^9(t^3+3t^2-9t-3)^3}
\]
for some $t\in\Q$.
\end{itemize}
\end{lemma}
\begin{proof}
It is easy to verify that both $H$ and the full inverse image of the split Cartan subgroup $C$ of $\GL_2(\Z/3\Z)$ in $\GL_2(\Z/9\Z)$ are of generalized $S_3$-type; it follows that if the image of $\rho_{E,3}$ lies in $C$ or if the image of $\rho_{E,9}$ lies in $H$, then $\rho_{E,9}$ gives an isomorphism from $\Gal(\Q(E[9])/\Q)$ to a group of generalized $S_3$-type and therefore $\Q(E[9])\subseteq \Q(3^\infty)$, so $E(\Q(3^\infty))[9]=E[9]$.

An enumeration of the subgroups $G\subseteq\GL_2(\Z/9\Z)$ of generalized $S_3$-type shows that either the image of $G$ in $\GL_2(\Z/3\Z)$ is conjugate to a subgroup of $C$, or $G$ is conjugate to a subgroup of $H$.
The groups $C$ and $H$ correspond to the congruence subgroups \href{http://www.uncg.edu/mat/faculty/pauli/congruence/csg0.html#group3D0}{$\textrm{3D}^0$} and \href{http://www.uncg.edu/mat/faculty/pauli/congruence/csg0.html#group9J0}{$\textrm{9J}^0$}, both of genus $0$; the rational maps from $X_C$ and $X_H$ to the $j$-line are taken from \cite{SZ}. 
\end{proof}

\begin{example}\label{ex-3max}
The elliptic curve $E/\Q$ with Cremona label $\href{http://www.lmfdb.org/EllipticCurve/Q/27a3}{\texttt{27a3}}$ admits two rational 3-isogenies, hence $E(\Q(3^\infty))(3)\simeq\Z/9\Z\oplus\Z/9\Z$.
On the other hand, the curve \href{http://www.lmfdb.org/EllipticCurve/Q/17100g2}{\texttt{17100g2}} admits only one rational 3-isogeny but also has $E(\Q(3^\infty))(3)\simeq \Z/9\Z\oplus\Z/9\Z$.
\end{example}


\begin{lemma}\label{lem-3x9}
Let $E/\Q$ be an elliptic curve.
Then $E(\Q(3^\infty))(3)\simeq \Z/3\Z\oplus\Z/9\Z$ if and only if the image of $\rho_{E,9}$ in $\GL_2(\Z/9\Z)$ is not of generalized $S_3$-type and is conjugate in $\GL_2(\Z/9\Z)$ to a subgroup of one of the following two groups:
\[
H_1:=\left\langle
\begin{pmatrix}1&1\\0&1\end{pmatrix},
\begin{pmatrix}2&0\\0&1\end{pmatrix},
\begin{pmatrix}2&0\\0&2\end{pmatrix}
\right\rangle,\qquad
H_2:=\left\langle
\begin{pmatrix}1&2\\3&1\end{pmatrix},
\begin{pmatrix}2&0\\0&1\end{pmatrix},
\begin{pmatrix}2&0\\0&2\end{pmatrix}
\right\rangle.
\]
Equivalently, $j(E)$ lies in the image of one of the rational maps
\[
j_1(t)=\frac{(t+3)^3(t^3+9t^2+27t+3)^3}{t(t^2+9t+27)},\qquad j_2(t)=\frac{(t+3)(t^2-3t+9)(t^3+3)^3}{t^3}.
\]
\end{lemma}
\begin{proof}
It is easy to verify that neither $H_1$ nor $H_2$ are of generalized $S_3$-type (which rules out $E(\Q(3^\infty))(3)\simeq \Z/9\Z\oplus\Z/9\Z$), and that each contains a normal subgroup $N_i$ for which the quotient $H_i/N_i$ is of generalized $S_3$-type, and for which the image of $N_i$ in $\GL_2(\Z/3\Z)$ is trivial and for which $N_i$ acts trivially on an element of order 9 in $\Z/9\Z\oplus\Z/9\Z$.
This implies that if $\Gal(\Q(E[9])/\Q)\simeq H_i$ then the base change of $E$ to the field $K_i\subseteq \Q(3^\infty)$ corresponding to the normal subgroup of $\Gal(\Q(E[9])/\Q)$ isomorphic to $N_i$ has torsion subgroup that contains a subgroup isomorphic to $\Z/3\Z\oplus\Z/9\Z$; moreover, $E(Q(3^\infty))(3)$ cannot be any larger than this because we have ruled out any points of order 27 in $E(\Q(3^\infty))$ (Proposition~\ref{prop-no27}) and $N_i$ cannot be the trivial group.

An enumeration of the subgroups of $\GL_2(\Z/9\Z)$ shows that every group $G$ that is not of generalized $S_3$-type and which contains a normal subgroup $N$ satisfying all the properties of $N_i$ above is either conjugate to a subgroup of $H_1$ or $H_2$, or is conjugate to a subgroup of
\[
H_3:=\left\langle
\begin{pmatrix}1&1\\3&1\end{pmatrix},
\begin{pmatrix}2&0\\0&1\end{pmatrix},
\begin{pmatrix}2&0\\0&2\end{pmatrix}
\right\rangle,
\]
with congruence subgroup \href{http://www.uncg.edu/mat/faculty/pauli/congruence/csg1.html#group9A1}{$\textrm{9A}^1$}.
As computed by Rouse and Zureick-Brown (using the techniques of~\cite{RZB}), the corresponding modular curve $X_{H_3}$ has genus 1 and is isomorphic to the elliptic curve \href{http://www.lmfdb.org/EllipticCurve/Q/27a3}{\texttt{27a3}}, which has just 3 rational points; two of these are cusps, while the other corresponds to $j$-invariant 0.
But for every elliptic curve $E/\Q$ with $j$-invariant $0$, we have $E(\Q^\infty)(3)\simeq \Z/9\Z\oplus\Z/9\Z$ as can be verified by checking one example and applying Proposition~\ref{prop-twist}.

The groups $H_1$ and $H_2$ yield congruence subgroups \href{http://www.uncg.edu/mat/faculty/pauli/congruence/csg0.html#group9B0}{$\textrm{9B}^0$} and
\href{http://www.uncg.edu/mat/faculty/pauli/congruence/csg0.html#group9C0}{$\textrm{9C}^0$}, respectively, both of genus zero; the maps $j_1(t)$ and $j_2(t)$ to the $j$-line are taken from \cite{SZ}.
\end{proof}

\subsection{Proof of Theorem \ref{thm-upperbound}}
Let $E/\Q$ be an elliptic curve.
Proposition \ref{prop-possible_primes} shows that any prime divisor~$p$ of the order of $E(\Q(3^\infty))_\tor$ lies in the set $\{2,3,5,7,13\}$.
Lemma \ref{lem-2primary} ($p=2$), Lemma~\ref{lem-9x27} and Proposition \ref{prop-no27} ($p=3$), Lemma \ref{lem-5tor} ($p=5$), Lemma \ref{lem-7tor} ($p=7$), and Lemma \ref{lem-13tor} ($p=13$) together imply that $E(\Q(3^\infty))_\tor$ is isomorphic to a subgroup of
\[
T_{\rm max} = (\Z/16\Z\oplus\Z/8\Z)\oplus(\Z/9\Z\oplus\Z/9\Z)\oplus\Z/5\Z\oplus(\Z/7\Z\oplus\Z/7\Z)\oplus\Z/13\Z.
\]
Examples \ref{ex-2max}, \ref{ex-3max}, \ref{ex-5max}, \ref{ex-7max}, \ref{ex-13max} for $p=2,3,5,7,13$, respectively, show that $T_{\rm max}$ is the smallest group with this property.\qed

\subsection{An algorithm to compute the structure of \texorpdfstring{$E(\Q(3^\infty))_\tor$}{\textit{E}Q(3*))tors}}
\label{sec:algorithm}

With Theorem~\ref{thm-upperbound} in hand we can now sketch a practical algorithm to compute the isomorphism type of $E(\Q(3^\infty))_\tor$ for a given elliptic curve $E/\Q$, which we may assume is defined by $y^2=f(x)$.  The strategy is to separately compute each $p$-primary component $E(\Q(3^\infty))(p)$ for $p=2,3,5,7,13$ by first determining the largest integer $k$ for which $E(\Q(3^\infty))[p^k]=E[p^k]$ and then determining the largest integer $j$ for which $E(\Q(3^\infty))(p)$ contains a point of order $p^j$.

Both steps make use of the division polynomials $f_{E,n}(x)$ whose roots are the distinct $x$-coordinates of the nonzero points $P\in E[n]$.
The polynomials $f_{E,n}(x)$ satisfy well-known recurrence relations that allow them to be efficiently computed; see \cite{mckee}, for example.
If $m$ divides $n$ then $f_{E,n}$ is necessarily divisible by $f_{E,m}$, and roots of the polynomial $f_{E,n}/f_{E,m}$ are the distinct $x$-coordinates of the points in $E[n]$ that do not lie in $E[m]$; by removing the factor $f_{E,m}$ of $f_{E,n}$ for each maximal proper divisor $m$ of~$n$ one obtains a polynomial $h_{E,n}$ whose roots are the distinct $x$-coordinates of the points in $E[n]$ of order~$n$.

The field $\Q(E[n])$ is an extension of the splitting field $K_f$ of $f_{E,n}(x)$ of degree at most~2 (the degree is 2 when $\im \rho_{E,n}$ contains $-1\in \GL_2(\Z/n\Z)$, and 1 otherwise, see \cite[Lemma~5.17]{sutherland2}). A necessary and sufficient condition for $\Q(E[n]))\subseteq \Q(3^\infty)$ is that for every irreducible factor $g$ of $h_{E,n}(x)$ with splitting field $K_g$, the field $L_g:=K_g(\sqrt{f(r)})$ is of generalized $S_3$-type, where $f(x)$ is the cubic defining $E:y^2=f(x)$ and $r$ is any root of $g$; note that each field $L_g$ is of the form $\Q(P)$ for some $P\in E[n]$ of order $n$ and is necessarily a Galois extension of $\Q$ that contains the coordinate of every point in $\langle P\rangle$.
A necessary and sufficient condition for $E(\Q(3^\infty))_\tor$ to contain a point of order $n$ is that for some irreducible factor $g$ of $h_{E,n}(x)$ the field $L_g$ is of generalized $S_3$-type.

We may thus compute $E(\Q(3^\infty))(p)$ as follows:

\begin{itemize}
\item Determine the largest $k$ for which $E[p^k]\subseteq \Q(3^\infty)$ by checking increasing values of $k$ from 1 up to the bound given by Theorem~\ref{thm-upperbound}.  For each $k$, compute the polynomial $h_{E,p^k}$, factor it over $\Q$, and for each irreducible factor $g$ compute the field $L_g$ and check whether $\Gal(L_g/\Q)$ is of generalized $S_3$-type (via Lemma~\ref{lem-s3char}) for all $g$.
\smallskip

\item Determine the largest $j$ for which $E(\Q(3^\infty))(p)$ contains a point of order $p^j$ by checking increasing values of $j$ from $k$ up to the bound given by Theorem~\ref{thm-upperbound}.  For each $k$, compute the polynomial $h_{E,p^j}$, factor it over $\Q$, and for each irreducible factor $g$ compute the field $L_g$ and check whether $\Gal(L_g/\Q)$ is of generalized $S_3$-type for some $g$.
\end{itemize}

As written this algorithm is not quite practical, but there are two things that may be done to make it so.
First, one can use a Monte Carlo algorithm to quickly rule out polynomials $g$ whose splitting fields cannot be of generalized $S_3$-type by picking random primes and factoring the reduced polynomial $\bar g$ over the corresponding finite field; if $\bar g$ has an irreducible factor whose degree does not divide 6 then the splitting field of $g$ cannot be of generalized $S_3$-type.  The second practical improvement is to use the explicit criterion for $j(E)$ given by Lemmas~\ref{lem-7tor},~\ref{lem-9x9}, and~\ref{lem-3x9} to more quickly compute the 3-primary and 7-primary components of $E(\Q(3^\infty))_\tor$.\footnote{We did not exploit this second improvement when using the algorithm to perform any of the explicit computations of $E(\Q(3^\infty))_\tor$ cited in \S\ref{sec-max}, since this improvement depends on some of these computations.}

A magma script implementing the algorithm with these optimizations can be found in \cite{magmascripts}; it was used to determine the 20 examples of minimal conductor that appear in Remark~\ref{rem-examples}.
These examples prove that each of these cases arise; in the next section we prove that no others do.

\begin{remark}
In Section~\ref{sec-param} we obtain a complete list of parameterizations for each torsion structure $E(\Q(3^\infty))_\tor$; see Table~\ref{table-param}.
With this list in hand one can immediately determine $E(\Q(3^\infty))_\tor$ from $j(E)$ whenever $j(E)\ne 1728$, making it  unnecessary to use the algorithm sketched above, except for distinguishing the two possibilities when $j(E)=1728$ (see Remark~\ref{rem-twist}).
However, the algorithm is implicitly used in several of the proofs in the next section that require us to explicitly check a finite number of cases, and our list of parameterizations depends on these results.  (We did not use the algorithm to prove any of the results in this section; see \cite{magmascripts} for details of our computations.)
\end{remark}

\section{The Structure of \texorpdfstring{$E(\Q(3^\infty))_\tor$}{\textit{E}Q(3*))tors}}

In this section we complete the classification of the torsion structures $T\simeq E(\Q(3^\infty))_\tor$ that appear in Theorem \ref{thm-main}.
There are a total of 1008 isomorphism types $T$ given by subgroups of the maximal group $T_{\max}$ that appears in Theorem~\ref{thm-upperbound}, of which 648 contain the minimal subgroup $\Z/2\Z\oplus\Z/2\Z$ required by Lemma \ref{lem-2tor}, but we will prove that in fact only 20 occur as $E(\Q(3^\infty))_\tor$ for some elliptic curve $E/\Q$.
In the five subsections that follow, for $p=13$, 7, 5, 3, 2, we will prove that there are 1, 4, 2, 5, 8 (respectively) possibilities for $T$ when $p$ is the largest prime divisor of its cardinality, and determine these $T$ explicitly.

We begin with a lemma that allows us to distinguish the two possibilities for $E(\Q(2^\infty))(2)$ permitted by Lemma~\ref{lem-2primary} when $E(\Q)[2]$ is trivial.
For an elliptic curve $E/\Q$, we use $\Delta(E)\in \Q^\times$ to denote its discriminant.
We recall that for $j(E)\ne 0,1728$, the image of $\Delta(E)$ in $\Q^\times/\Q^{\times 2}$ is determined by $j(E)$ (see \cite[Cor. 5.4.1]{silverman}); in fact,
\begin{equation}\label{eq-dsq}
\Delta(E)\equiv j(E)-1728\quad (\text{in }\Q^\times/\Q^{\times 2}),
\end{equation}
as one can verify by computing the discriminant $\Delta(E)=-16(4A^3+27B^2)$ of the elliptic curve $E:y^2=x^3+Ax+B$ with $A=3j(E)(1728-j(E))$ and $B=2j(E)(1728-j(E))^2$ both nonzero.

\begin{lemma}\label{lem-full4tor}
Let $E/\Q$ be an elliptic curve for which $E(\Q)[2]$ is trivial, but $E(\Q(3^\infty))[4]=E[4]$. Then $-\Delta(E)$ is a square in $\Q$ and
\[
j(E) = \frac{-4(t^2-3)^3(t^2 - 8t - 11)}{(t+1)^4},
\]
for some $t \in \Q \setminus \{-1\}$.
\end{lemma}

\begin{proof}
If $E(\Q)[2]$ is trivial and $E(\Q(3^\infty))[4]=E[4]$ then the image $G:=\im \rho_{E,4}$ is conjugate to a subgroup of $\GL_2(\Z/4\Z)$ of generalized $S_3$-type whose image in $\GL_2(\Z/2\Z)$ does not fix any nonzero element of $\Z/2\Z\oplus\Z/2\Z$ (equivalently, has order at least 3).
As noted in \S 2, the group $G$ must have a surjective determinant map and contain an element $\gamma$ corresponding to complex conjugation (here we use the stronger criterion of \cite[Rem. 3.15]{sutherland2}).
An enumeration of the subgroups of $\GL_2(\Z/4\Z)$ finds that every such~$G$ is conjugate to a subgroup of
\[
H:=\left\langle
\begin{pmatrix}3&1\\0&1\end{pmatrix},
\begin{pmatrix}0&3\\1&3\end{pmatrix},
\begin{pmatrix}3&0\\0&3\end{pmatrix}
\right\rangle.
\]
The corresponding modular curve $X_H$ is labeled $\texttt{X20a}$ in \cite{RZB} and has genus zero.
A map to the $j$-line is given by the rational function
\[
j(t) := \frac{-4(t^2-3)^3(t^2 - 8t - 11)}{(t+1)^4}.
\]
Since neither $0$ nor $1728$ lie in the image of the map $j(t)$, from \eqref{eq-dsq} we see that the discriminant $\Delta(t)$ of an elliptic curve over $\Q$ with $j$-invariant $j(t)$ must satisfy
\[
\Delta(t) \equiv j(t)-1728\equiv -1\quad (\text{in } \Q^\times/\Q^{\times 2}),
\]
thus $-\Delta(t)$ is always a square, as claimed.
\end{proof}

\subsection{When 13 divides \texorpdfstring{$\#\boldsymbol{E(\Q(3^\infty))_\tor}$}{\#\textit{E}(Q(3*))tors}}

There is only one possibility for $E(\Q(3^\infty))_\tor$ when it contains a point of order $13$.

\begin{prop}\label{prop-13div}
Let $E/\Q$ be an elliptic curve for which $E(\Q(3^\infty))_\tor$ contains a point of order $13$.
Then $E(\Q(3^\infty))_\tor$ is isomorphic to $\Z/2\Z\oplus \Z/26$.
\end{prop}

\begin{proof}
By Lemma \ref{lem-13tor}, $E$ must admit a rational $13$-isogeny, since $E(\Q(3^\infty))(13)$ is non-trivial.
Theorem~\ref{thm-isog} implies that $E$ admits no other rational $n$-isogenies, and it follows that $\Q(3^\infty)(3)$, $\Q(5^\infty)(5)$, and $\Q(7^\infty)(7)$ are all trivial, by Lemma \ref{lem-3tor}, Lemma \ref{lem-5tor}, and Lemma \ref{lem-7tor} and Corollary~\ref{cor-14x14}, respectively.
Since $E$ admits no rational $2$-isogenies, $E(\Q)[2]$ is trivial, and Lemma \ref{lem-2primary} implies that $E(\Q(3^\infty))(2)$ is isomorphic to either $E[2]$ or $E[4]$.
By Lemma \ref{lem-full4tor}, if the latter holds then $-\Delta(E)$ is a rational square; we claim that this cannot occur.

The modular curve $X_0(13)$ that parameterizes 13-isogenies has genus 0 and yields a rational parameterization of the $j$-invariants of elliptic curves $E/\Q$ that admit a rational 13-isogeny.  From \cite[Table 3]{lozano1} we see that $j(E)$ must lie in the image of the rational map
\[
j(t) := \frac{(t^2+5t+13)(t^4+7t^3+20t^2+19t+1)^3}{t}.
\]
Neither $0$ nor $1728$ lie in the image of the map $j(t)$, so by \eqref{eq-dsq}, the corresponding
discriminant $\Delta(t)$ of an elliptic curve over $\Q$ with $j$-invariant $j(t)$ must satisfy
\[
\Delta(t) \equiv (j(t)-1728)^3\equiv t(t^2 + 6t + 13)\quad (\text{in }\Q^\times/\Q^{\times 2}),
\]
with $t\ne 0$.
Finding $t\in\Q^\times$ for which $-\Delta(t)\in \Q$ is a square is equivalent to finding nonzero rational points~$P$ on the elliptic curve
\[
E_\Delta\colon y^2 = x(x^2 - 6x + 13)
\]
for which $x(P)\ne 0$, equivalently, $P\not\in E_\Delta(\Q)[2]$.
But a calculation shows that $E_\Delta$ has rank~$0$ and torsion subgroup isomorphic to $\Z/2\Z$, so no such $P$ exists.
\end{proof}

\begin{remark}\label{rem-13inf}
One can obtain infinitely many elliptic curves $E/\Q$ with $E(\Q(3^\infty))_\tor\simeq \Z/2\Z\oplus \Z/26$ and distinct $j$-invariants by choosing $E$ for which $E(F)\simeq \Z /13 \Z$ for some cubic field $F$, as shown in \cite{najman}.
The curve \href{http://www.lmfdb.org/EllipticCurve/Q/147b1}{\texttt{147b1}} is an example  $F=\Q[x]/(x^3+x^2-2x-1)$.
\end{remark}

\subsection{When 7 divides \texorpdfstring{$\#\boldsymbol{E(\Q(3^\infty))_\tor}$}{\#\textit{E}(Q(3*))tors}}

We now address the cases where $\#E(\Q(3^\infty))_\tor$ is divisible by 7 (but not 13).
The case where it is also divisible by 49 is already covered by Lemma~\ref{lem-7tor}, and Corollary~\ref{cor-14x14}, which imply that we then must have $E(\Q(3^\infty))_\tor\simeq \Z/14\Z\oplus\Z/14\Z$.
Theorem~\ref{thm-isog} and Lemma~\ref{lem-7tor} then leave us just 3 possibilities to consider: (1) $E$ admits a rational 21-isogeny, (2)~$E$ admits a rational 14-isogeny, (3) $E$ admits a rational 7-isogeny and no others.
These are addressed in the next  three lemmas.
Recall that if $E$ admits a rational $m$-isogeny $\varphi$ and a rational $n$-isogeny $\psi$, with $m$ and $n$ coprime, then it necessarily admits a rational $mn$-isogeny, namely, the isogeny $E\to E/\langle \ker \varphi,\ker \psi\rangle$.

\begin{lemma}\label{lem-21isog}
Let $E/\Q$ be an elliptic curve.
Then $E$ admits a rational $21$-isogeny if and only if $E(\Q(3^\infty))_\tor\simeq \Z/6\Z\oplus\Z/42\Z$.
\end{lemma}

\begin{proof}
It follows from Lemmas~\ref{lem-7tor} and \ref{lem-3tor} that if $E(\Q(3^\infty))_\tor\simeq \Z/6\Z\oplus\Z/42\Z$, then $E$ admits a rational 7-isogeny and a rational 3-isogeny, hence a rational 21-isogeny.
From \cite[Table~4]{lozano1} we see that there are just four $\Qbar$-isomorphism classes of elliptic curves $E/\Q$ that admit a rational 21-isogeny, represented by the four elliptic curves in the isogeny class with Cremona label \href{http://www.lmfdb.org/EllipticCurve/Q/162b}{\texttt{162b}}.
A direct computation finds that $E(\Q(3^\infty))_\tor\simeq \Z/6\Z\oplus\Z/42\Z$ for each of these four curves.
\end{proof}

\begin{lemma}\label{lem-14isog}
Let $E/\Q$ be an elliptic curve.  If $E$ admits a rational $14$-isogeny then $E(\Q(3^\infty))_\tor$ is isomorphic to $\Z/2\Z\oplus\Z/14\Z.$
\end{lemma}
\begin{proof}
From \cite[Table~4]{lozano1} we see that there are just two $\Qbar$-isomorphism classes of elliptic curves $E/\Q$ that admit a rational 14-isogeny, represented by the curves $\href{http://www.lmfdb.org/EllipticCurve/Q/49/a//4}{\texttt{49a1}}$ and $\href{http://www.lmfdb.org/EllipticCurve/Q/49/a/3}{\texttt{49a2}}$.
A direct computation finds that $E(\Q(3^\infty))_\tor\simeq\Z/2\Z\oplus\Z/14\Z$ for both curves.
\end{proof}

\begin{lemma}\label{lem-one7isog}
Let $E/\Q$ be an elliptic curve.
If $E$ admits a rational $7$-isogeny and no other non-trivial rational $n$-isogenies, then $E(\Q(3^\infty))_\tor$ is isomorphic to $\Z/2\Z \oplus \Z/14\Z$ or $\Z/4\Z \oplus \Z/28\Z$.
\end{lemma}
\begin{proof}
Lemmas \ref{lem-3tor}, \ref{lem-5tor}, and \ref{lem-13tor} imply that $E(\Q(3^\infty))(p)$ is trivial for $p=3,5,13$, and Lemma \ref{lem-2primary} implies that $E(\Q(3^\infty))(2)=E[2]$ or $E[4]$.
\end{proof}

We summarize the results of this subsection in the following proposition.

\begin{prop}\label{prop-7div}
Let $E/\Q$ be an elliptic curve for which $E(\Q(3^\infty))_\tor$ contains a point of order~$7$.
Then $E(\Q(3^\infty))_\tor$ is isomorphic to one of the groups: $\Z/2\Z\oplus\Z/14\Z,\ \Z/4\Z\oplus\Z/28\Z,\ \Z/6\Z\oplus \Z/42\Z$, $\Z/14\Z \oplus \Z/14\Z$.
\end{prop}

\begin{proof}
This follows from Corollary~\ref{cor-14x14} and Lemmas~\ref{lem-21isog}, \ref{lem-14isog}, \ref{lem-one7isog}.
\end{proof}

\subsection{When 5 divides \texorpdfstring{$\#\boldsymbol{E(\Q(3^\infty))_\tor}$}{\#\textit{E}(Q(3*))tors}}

We now address the cases where $\#E(\Q(3^\infty))_\tor$ is divisible by 5 (but not 7 or 13).

\begin{lemma}\label{lem-15isog}
Let $E/\Q$ be an elliptic curve.
If $E$ admits a rational $15$-isogeny then  $E(\Q(3^\infty))_\tor$ is isomorphic to $\Z/6\Z\oplus\Z/6\Z$ or $\Z/6\Z\oplus \Z/30\Z$ (both occur).
If $E(\Q(3^\infty))_\tor\simeq \Z/6\Z\oplus\Z/30\Z$ then $E$ admits a rational $15$-isogeny.
\end{lemma}
\begin{proof}
As can be seen in \cite[Table 4]{lozano1}, there are four $\Qbar$-isomorphism classes of elliptic curves $E/\Q$ that admit a rational $15$-isogeny, represented by the four curves in isogeny class \href{http://www.lmfdb.org/EllipticCurve/Q/50a}{\texttt{50a}}.
A direct computation finds that $E(\Q(3^\infty))_\tor\simeq \Z/6\Z\oplus \Z/6\Z$ for the curves
\href{http://www.lmfdb.org/EllipticCurve/Q/50a1}{\texttt{50a1}} and
\href{http://www.lmfdb.org/EllipticCurve/Q/50a2}{\texttt{50a2}},
while $E(\Q(3^\infty))_\tor\simeq \Z/6\Z\oplus \Z/30\Z$ for the curves
\href{http://www.lmfdb.org/EllipticCurve/Q/50a3}{\texttt{50a3}} and
\href{http://www.lmfdb.org/EllipticCurve/Q/50a4}{\texttt{50a4}}.
It follows from Lemmas~\ref{lem-5tor} and~\ref{lem-3tor} that if  $E(\Q(3^\infty))_\tor\simeq \Z/6\Z\oplus \Z/30\Z$ then $E$ admits a rational 5-isogeny  and a rational 3-isogeny, hence a rational 15-isogeny.
\end{proof}

\begin{prop}\label{prop-5div}
Let $E/\Q$ be an elliptic curve for which $E(\Q(3^\infty))$ contains a point of order $5$
Then $E(\Q(3^\infty))_\tor$ is isomorphic to $\Z/2\Z\oplus\Z/10\Z$ or $\Z/6\Z\oplus\Z/30\Z$.
\end{prop}
\begin{proof}
As noted above, the results of the previous two subsections imply that $E(\Q(3^\infty))_\tor$ is not divisible by 7 or 13.
Lemma \ref{lem-5tor} implies that $E$ admits a rational 5-isogeny, and if $E(\Q(3^\infty))(3)$ is non-trivial, then $E$ also admits a rational $3$-isogeny, by Lemma \ref{lem-3tor}, in which case it falls into the case covered by Lemma \ref{lem-15isog}.
We know that $E(\Q(3^\infty))(5)\simeq \Z/5\Z$, by Lemma~\ref{lem-5tor}, thus it remains only to consider $E(\Q(3^\infty))(2)$ when $E(\Q(3^\infty))(p)$ is trivial for $p=3,7,13$.

We first suppose that $E(\Q)[2]$ is non-trivial.
Then $E(\Q(3^\infty))(2)=E(\Q(2^\infty))(2)$, by Lemma \ref{lem-2primary}.
Lemma \ref{lem-5tor} implies that $E(\Q(3^\infty))(5)=E(\Q(2^\infty))(5)$, since $E$ must admit a rational 5-isogeny whose kernel generates and extension of degree at most 2, hence a subfield of $\Q(2^\infty)$.
Theorem \ref{thm-fujita} then implies $E(\Q(3^\infty))_\tor=E(\Q(2^\infty))_\tor\simeq \Z/2\Z\oplus\Z/10\Z$.

We now suppose that $E(\Q)[2]$ is trivial.
Then $E(\Q(2^\infty))$ is trivial and $E(\Q(3^\infty))=E[2]$ or $E[4]$, by Lemma \ref{lem-2primary}.
Lemma \ref{lem-full4tor} implies that the latter holds only when $-\Delta(E)$ is a rational square.
We claim that this cannot occur.
From \cite[Table 3]{lozano1}, we see that since $E$ admits a rational 5-isogeny, its $j$-invariant must lie in the image of the rational map
\[
j(t) = \frac{(t^2+10t+5)^3}{t}.
\]
Neither $0$ nor $1728$ lie in the image of this map, so by \eqref{eq-dsq}, the discriminant $\Delta(t)$ of an elliptic curve over $\Q$ with $j$-invariant $j(t)$ must satisfy
\[
\Delta(t) \equiv (j(t)-1728)^3 \equiv t(t^2 + 22t + 125)\quad (\text{in }\Q^\times/\Q^{\times 2}),
\]
with $t\ne 0$.
Finding $t\in \Q^\times$ for which $-\Delta(t)$ is a square is equivalent to finding rational points $P$ on the elliptic curve
\[
E_\Delta\colon y^2 = x(x^2 - 22x + 125)
\]
that do not lie in $E_\Delta(\Q)[2]$.
But we find that $E_\Delta(\Q)\simeq\Z/2\Z$, so no such $P$ exist.
Thus we must have $E(\Q(3^\infty))(2)=E[2]$, and therefore $E(\Q(3^\infty))_\tor\simeq \Z/2\Z\oplus\Z/10\Z$.
\end{proof}

\subsection{When only 2 and 3 divide \texorpdfstring{$\#E(\Q(3^\infty))$}{\#\textit{E}(Q(3*))}}

We now consider the case where $\#E(\Q(3^\infty))_\tor$ is divisible by 3 but not by 5, 7, or 13.
Lemmas~\ref{lem-2tor} and ~\ref{lem-3tor} imply $E(\Q(3^\infty))[6]=E[6]$, thus if $E(\Q(3^\infty))$ does not cannot contain any points of order $24$ or $36$, then Theorem~\ref{thm-upperbound} implies that $E(\Q(3^\infty))_\tor$ must be isomorphic to one of the five groups
\begin{equation}\label{eq-3list}
\Z/6\Z\oplus \Z/6\Z,\ \ \Z/6\Z\oplus \Z/12\Z,\ \ \Z/6\Z\oplus \Z/18\Z,\ \ \Z/12\Z \oplus \Z/12\Z,\ \ \Z/18\Z \oplus \Z/18\Z.
\end{equation}
As shown by the examples in Remark \ref{rem-examples}, these cases all occur for some $E/\Q$, so it suffices to show that $E(\Q(3^\infty))$ cannot contain any points of order 24 or 36.

\begin{proposition}\label{prop-no24}
Let $E/\Q$ be an elliptic curve.  There are no points of order $24$ in $E(\Q(3^\infty))$.
\end{proposition}
\begin{proof}
Suppose $E(\Q(3^\infty))$ contains a point of order 24; then it contains both a point of order 3 and a point of order 8.
Lemma~\ref{lem-3tor} implies that $E$ admits a rational 3-isogeny, and the points in the kernel of this 3-isogeny are defined over a quadratic extension (by Lemma~\ref{lem-isog_gal}), so $E(\Q(2^\infty))$ contains a point of order 3.
Lemma~\ref{lem-2primary} implies that $E(\Q)[2]$ is non-trivial and  $E(\Q(3^\infty))(2)=E(\Q(2^\infty))(2)$, so $E(\Q(2^\infty))$ contains a point of order~8.
But then $E(\Q(2^\infty))$ contains a point of order 24, which contradicts Theorem \ref{thm-fujita}.
\end{proof}

In order to rule out a point of order 36 in $E(\Q(3^\infty))$ we require the following lemmas.

\begin{lemma}\label{lem-4isog}
Let $E/\Q$ be an elliptic curve. If $E(\Q)$ contains a point of order $2$, and $E(\Q(3^\infty))$ contains a point of order $4$, then either $E(\Q)[2]=E[2]$ or $E$ admits a rational $4$-isogeny.
\end{lemma}
\begin{proof}
It suffices to consider the possible images $G\subseteq \GL_2(\Z/4\Z)$ of $\rho_{E,4}$.
An enumeration of the subgroups $G$ of $\GL_2(\Z/4\Z)$ finds that whenever the image of $G$ in $\GL_2(\Z/2\Z)$ fixes a nonzero element of $\Z/2\Z\oplus\Z/2\Z$ (i.e. $E(\Q)$ contains a point of order 2) and $G$ contains a normal subgroup~$N$ for which $G/N$ is of generalized $S_3$-type and $N$ fixes an element of order~4 in $\Z/4\Z\oplus\Z/4\Z$ (i.e. $E(\Q(3^\infty))$ contains a point of order 4), then either the image of $G$ in $\GL_2(\Z/2\Z)$ is trivial ($E(\Q)[2]=E[2]$) or $G$ stabilizes a cyclic subgroup of $\Z/4\Z\oplus\Z/4\Z$ of order 4 ($E$ admits a rational 4-isogeny).
\end{proof}

\begin{lemma}\label{lem-9isog-no4}
Let $E/\Q$ be an elliptic curve that admits a rational $9$-isogeny.
Then $E(\Q(3^\infty))$ does not contain a point of order $4$.
\end{lemma}
\begin{proof}
If $E(\Q)[2]=E[2]$ then $E$ is isogenous to an elliptic curve that admits a rational 4-isogeny and a rational 9-isogeny, hence a rational 36-isogeny, which is ruled out by Theorem \ref{thm-isog}.
If $E(\Q)[2]$ has order $2$ then $E(\Q(3^\infty))$ cannot contain a point of order~4, because $E$ would then admit a rational 4-isogeny, by Lemma \ref{lem-4isog}, hence a rational 36-isogeny, which is again ruled out by Theorem \ref{thm-isog}.

We are thus left to consider the possibility that $E(\Q)[2]$ is trivial and $E(\Q(3^\infty))$ has a point of order 4,
in which case Lemma \ref{lem-2primary} implies $E(\Q(3^\infty))[4]=E[4]$, and Lemma~\ref{lem-full4tor} implies that $-\Delta(E)$ is a square.
We can assume $j(E)\ne 0$ because a direct computation shows that for the curve \href{http://www.lmfdb.org/EllipticCurve/Q/27a3}{\texttt{27a3}} with $j(E)=0$ we have $E(\Q(3^\infty))_\tor\simeq \Z/18\Z\oplus\Z/18\Z$, which does not contain a point of order 4.
Proposition~\ref{prop-twist} implies that this is true for every $E/\Q$ with $j(E)=0$.

From \cite[Table 3]{lozano1} we see that $j(E)$ must lie in the image of the rational map
\[
j(t)=\frac{t^3(t^3-24)^3}{t^3-27}.
\]
Having ruled out $j(E)=0$, we can assume $j(t)\ne 0$ (so $t\ne 0$), and 1728 does not lie in the image of $j(t)$, so by \eqref{eq-dsq}, for any $t\ne 0,3$ the discriminant $\Delta(t)$ of an elliptic curve with $j$-invariant $j(t)$
satisfies
\[
\Delta(t)\equiv(j(t)-1728)^3\equiv (t-3)(t^2+3t+9)\quad (\text{in } \Q^\times/\Q^{\times 2}).
\]
To see whether $-\Delta(t)$ can be square when $t\ne 0,3$, we search for nonzero rational points $P$ with $x(P)\neq 0,3$ on the elliptic curve
\[
E_\Delta\colon y^2=(x+3)(x^2-3x+9).
\]
We find that $E_\Delta(\Q) \simeq \Z/2\Z$, and the nonzero rational point has $x$-coordinate $3$.
Thus no such $P$ exist and the lemma follows.
\end{proof}

\begin{lemma}\label{lem-gal9image}
Suppose that $E/\Q$ admits just one rational $3$-isogeny and no rational $9$-isogenies, and that $E(\Q(3^\infty))$ contains a point of order $9$.  Then
\[
j(E)=\frac{(t+3)(t^2-3t+9)(t^3+3)^3}{t^3}
\]
for some $t\in \Q^\times$.
\end{lemma}
\begin{proof}
To determine the possible images of the mod-9 Galois representation of an elliptic curve $E/\Q$ satisfying the hypothesis of the proposition, we conducted a search similar to that used in the proof of Lemma \ref{lem-9x27}, using Magma to enumerate the subgroups of $\GL_2(\Z/9\Z)$ (up to conjugacy).
We find that $\rho_{E,9}(\Gal(\Q(E[9])/\Q)$ must be conjugate in $\GL_2(\Z/9\Z)$
to a subgroup of one of the groups
\begin{align*}
H_1&:=\left \langle
\begin{pmatrix} 1 & 0 \\ 0 & 2 \end{pmatrix},
\begin{pmatrix} 2 & 0 \\ 0 & 1 \end{pmatrix},
\begin{pmatrix} 1 & 3 \\ 0 & 1 \end{pmatrix},
\begin{pmatrix} 1 & 1 \\ 6 & 1 \end{pmatrix}
\right \rangle,\\
H_2&:=\left \langle
\begin{pmatrix} 1 & 0 \\ 0 & 2 \end{pmatrix},
\begin{pmatrix} 2 & 0 \\ 0 & 1 \end{pmatrix},
\begin{pmatrix} 1 & 3 \\ 0 & 1 \end{pmatrix},
\begin{pmatrix} 1 & 1 \\ 3 & 1 \end{pmatrix}
\right \rangle,
\end{align*}
whose intersections with $\SL_2(\Z/9\Z)$ yield the congruence subgroups \href{http://www.uncg.edu/mat/faculty/pauli/congruence/csg0.html#group9C0}{$\textrm{9C}^0$} and \href{http://www.uncg.edu/mat/faculty/pauli/congruence/csg1.html#group9A1}{$\textrm{9A}^1$}, of genus~0 and~1, respectively.
We will show that $H_2$ cannot occur unless $j(E)=0$, which we note is of the form required by the lemma (take $t=-3$); in fact, when $j(E)=0$ the image of $\rho_{E,9}$ is conjugate to a subgroup of~$H_1$ that may also lie in $H_2$ (this depends on $E$).

The intersection of $H_1$ and $H_2$ is the subgroup
\[
H_3 :=\left \langle
\begin{pmatrix} 1 & 0 \\ 0 & 2 \end{pmatrix},
\begin{pmatrix} 2 & 0 \\ 0 & 1 \end{pmatrix},
\begin{pmatrix} 1 & 3 \\ 0 & 1 \end{pmatrix}
\right \rangle,
\]
which is equal to the image of $\Gamma_0(3,9)$ in $\GL_2(\Z/9\Z)$; the modular curve $X_{H_3}=X_0(3,9)$ has genus~1 (it corresponds to the congruence subgroup \href{http://www.uncg.edu/mat/faculty/pauli/congruence/csg1.html#group9A1}{$\textrm{9A}^1$}), and parameterizes elliptic curves that admit a 3-isogeny and a 9-isogeny whose kernels intersect trivially.
The index-3 inclusion $H_3\subseteq H_2$ gives a degree-3 map $\varphi\colon X_{H_3}\to X_{H_2}$ of genus 1 curves, and a calculation using \cite[Lemma~3.4]{zywina2} shows that both curves have two rational cusps ($X_0(3,9)$ has six cusps in all, but only two are rational).
We may thus view the modular curves $X_{H_2}$ and $X_{H_3}$ as elliptic curves over $\Q$, and since $\varphi$ must map cusps to cusps, we can choose the origins so that $\varphi$ is an isogeny.
Both curves are defined over $\Q$ ($H_2$ and $H_3$ both have surjective determinant maps), so $\varphi$ is also defined over $\Q$; we thus have a rational 3-isogeny from $X_0(3,9)$ to $X_{H_2}$.

The elliptic curve corresponding to $X_{H_3}=X_0(3,9)$ has Cremona label \href{http://www.lmfdb.org/EllipticCurve/Q/27a1}{\texttt{27a1}}, and an examination of its isogeny class \href{http://www.lmfdb.org/EllipticCurve/Q/27/a}{\texttt{27a}} shows that $X_{H_2}$ is isomorphic to either \href{http://www.lmfdb.org/EllipticCurve/Q/27a2}{\texttt{27a2}} or \href{http://www.lmfdb.org/EllipticCurve/Q/27a3}{\texttt{27a3}}, and it must be the latter, since \href{http://www.lmfdb.org/EllipticCurve/Q/27a2}{\texttt{27a2}} has only one rational point but $X_{H_2}$ has two rational cusps.
The curve \href{http://www.lmfdb.org/EllipticCurve/Q/27a3}{\texttt{27a3}} isomorphic to $X_{H_2}$ has three rational points, so $X_{H_2}$ has exactly one non-cuspidal rational point, corresponding to the $\Qbar$-isomorphism class of an elliptic curve $E/\Q$ with $\im\rho_{E,9}\subseteq H_2$.
	
To determine this $\Qbar$-isomorphism class it suffice to find one representative.
The curve \href{http://www.lmfdb.org/EllipticCurve/Q/27a1}{\texttt{27a1}} itself admits a rational 3-isogeny and a rational 9-isogeny with distinct kernels and thus corresponds to a non-cuspidal rational point on $X_0(3,9)$, and its image under $\varphi$ is a non-cuspidal rational point on $X_{H_2}$.\footnote{This does not contradict the fact that \href{http://www.lmfdb.org/EllipticCurve/Q/27a1}{\texttt{27a1}} does not satisfy the hypothesis of Lemma~\ref{lem-gal9image}; elliptic curves whose mod-9 image is properly contained in $H_2$ may admit more than one rational 3-isogeny and/or a rational 9-isogeny.}
It follows that if $j(E)\ne 0$ then its mod-9 image must be conjugate to a subgroup of $H_1$.

From  the tables in \cite{SZ} we see that for the genus 0 curve $X_{H_1}$ the map to the $j$-line is given by
\[
j(t)=\frac{(t+3)(t^2-3t+9)(t^3+3)^3}{t^3},
\]
which is the function appearing in the statement of the lemma.
\end{proof}

\begin{example}\label{ex-gal9image}
The elliptic curve \href{http://www.lmfdb.org/EllipticCurve/Q/722a1}{\texttt{722a1}} satisfies the hypothesis of Lemma~\ref{lem-gal9image}:
it admits a single rational 3-isogeny but not a 9-isogeny, and has a point of order 9 over the compositum of the cubic fields of discriminant $361$ and $-1083$, hence over $\Q(3^\infty)$.
The image of $\rho_{E,9}$ is conjugate to~$G_1$, and we note that $j(E)=2375/8$ is of the form required by the lemma if we take $t=-2$.
\end{example}

\begin{lemma}\label{lem-two3isog-no4}
Let $E/\Q$ be an elliptic curve.
If $E$ admits more than one rational $3$-isogeny then $E(\Q(3^\infty))$ does not contain a point of order $4$.
\end{lemma}
\begin{proof}
If $E$ admits more than one rational 3-isogeny then it is related by a rational 3-isogeny $\varphi$ to an elliptic curve $E'/\Q$ that admits a rational $9$-isogeny.
The 3-isogeny $\varphi\colon E\to E'$ will map any point of order 4 in $E(\Q(3^\infty))$ to a point of order 4 in $E'(\Q(3^\infty))$, but no such point can exist, by Lemma \ref{lem-9isog-no4}.
\end{proof}

\begin{proposition}\label{prop-no36}
Let $E/\Q$ be an elliptic curve.
Then $E(\Q(3^\infty))$ contains no points of order~$36$.
\end{proposition}
\begin{proof}
Suppose for the sake of contradiction that $E(\Q(3^\infty))$ does contain a point of order~36.
It follows from Lemmas \ref{lem-3tor}, \ref{lem-9isog-no4} and \ref{lem-two3isog-no4} that $E$ admits exactly one rational $3$-isogeny and no rational $9$-isogenies.
We now consider two cases.

Let us first suppose that $E(\Q)[2]$ is trivial.
Since $\Q(3^\infty)$ contains a point of order 36, it contains a point of order 4, and Lemma \ref{lem-full4tor} implies that
\[
j(E) = \frac{-4(t^2-3)^3(t^2-8t-11)}{(t+1)^4},
\]
for some $t \in \Q \backslash \{-1\}$.
Since $E$ admits a rational $3$-isogeny, its $j$-invariant must also satisfy
\[
j(E) = \frac{(s+27)(s+3)^3}{s}
\]
for some $s\in \Q^\times$ (see \cite[Table 3]{lozano1}, for example).
The valid pairs $(t,s)$ lie on the (singular) curve
\[
C_1:-4s(t^2-3)^3(t^2 - 8t - 11)- (s+27)(s+3)^3(t+1)^4=0,
\]
which has genus 1 and the rational point $(0,-1)$.
Its normalization is isomorphic to the elliptic curve \href{http://www.lmfdb.org/EllipticCurve/Q/48a3}{\texttt{48a3}}, which has 8 rational points and is a smooth model for the modular curve $X_G$ obtained by taking the fiber product over $X(1)$ of the two maps above from the genus zero curves $X_H$ and $X_0(3)$ to $X(1)$; here $H$ is the group in the proof of Lemma~\ref{lem-full4tor} and $G$ is the intersection in $\GL_2(\Z/12\Z)$ of the inverse images of $H\subseteq\GL_2(\Z/4\Z)$ and the Borel group in $\GL_2(\Z/3\Z)$.
A calculation in Magma shows that $X_G$ has four rational cusps, and that the points
\[
 (-5,-36), (7,-81/4), (-5/4, -81/4), (-1/2 , -36) \in C_1(\Q),
\]
are valid solutions $(t,s)$ corresponding to the four non-cuspidal rational points on $X_G$.
These solutions yield two distinct $j$-invariants: $-35937/4$ and $109503/64$.
Taking the curves \href{http://www.lmfdb.org/EllipticCurve/Q/162a1}{\texttt{162a1}} and \href{http://www.lmfdb.org/EllipticCurve/Q/162d1}{\texttt{162d1}} as representatives of these $\Qbar$-isomorphism classes, we find that neither has a point of order~36 defined over $\Q(3^\infty)$, and by Proposition~\ref{prop-twist}, this applies to every $E/\Q$ in these two classes.

We now suppose that $E(\Q)[2]$ is non-trivial and proceed similarly.
Now $E$ has a rational point of order 2, so its $j$-invariant has the form
\[
j(E)=\frac{(s+256)^3}{s^2},
\]
for some $s\in \Q^\times$ (see \cite[Table 3.]{lozano1}, for example).
By Lemma \ref{lem-gal9image}, the $j$-invariant $j(E)$ also satisfies
\[
j(E)=\frac{(t+3)(t^2-3t+9)(t^3+3)^3}{t^3},
\]
for some $t\in \Q^\times$.
The possible solutions $(t,s)$ lie on the genus 2 curve
\[
C_2:(t+3)(t^2-3t+9)(t^3+3)^3s^2 - t^3(s+256)^3=0,
\]
which has the hyperelliptic model
\[
C_3:y^2 = x^6 - 34x^3 + 1.
\]
The Jacobian of $C_3$ has rank 0, and using Chabauty's method we find that
\[
C_3(\Q)=\left \{\pm \infty , (-1 , \pm 6), (0, \pm 1) \right \}.
\]
There are thus six rational points on the modular curve $X_G$ corresponding to the fiber product over $X(1)$ of the two rational maps from the genus zero curves $X_1(2)=X_0(2)$ and $X_{H_1}$, where $H_1$ is the group in the proof of \ref{lem-gal9image} and $G$ is the intersection in $\GL_2(\Z/18\Z)$ of the inverse images of the Borel group in $\GL_2(\Z/2\Z)$ and $H_1\subseteq \GL_2(\Z/9\Z)$.  A calculation in Magma shows that $X_G$ has four rational cusps, and that the points
\[
 (3,-16), (-3,-256) \in C_2(\Q)
\]
are valid solutions $(t,s)$ corresponding to the two non-cuspidal rational points on $X_G$, which yield the $j$-invariants $0$ and $54000$.
Taking the elliptic curves \href{http://www.lmfdb.org/EllipticCurve/Q/27a}{\texttt{27a1}} and \href{http://www.lmfdb.org/EllipticCurve/Q/36a2}{\texttt{36a2}} as representatives of these $\Qbar$-isomorphism classes, we find that neither has a point of order 36 defined over $\Q(3^\infty)$.
\end{proof}

\begin{corollary}\label{cor-3div}
Let $E/\Q$ be an elliptic curve.  If $3$ is the largest prime divisor of $\#E(\Q(3^\infty))_\tor$ then $E(\Q(3^\infty))_\tor$ is isomorphic to one of the five groups listed in \eqref{eq-3list}.
\end{corollary}
\begin{proof}
As argued at the start of this subsection, this now follows from Propositions~\ref{prop-no24} and~\ref{prop-no36}.
\end{proof}

\subsection{When only 2 divides \texorpdfstring{$\boldsymbol{\#E(\Q(3^\infty))_\tor}$}{\#\textit{E}(Q(3*))}}

If $\#E(\Q(3^\infty))$ is a power of $2$ then Lemmas~\ref{lem-2tor} and~\ref{lem-2primary} imply that
\[
E(\Q(3^\infty))\simeq\begin{cases}
\Z/2\Z\oplus \Z/2^j\Z & j=1,2,3,4,\text{ or}\\
\Z/4\Z\oplus \Z/2^j\Z & j=2,3,4,\text{ or}\\
\Z/8\Z\oplus \Z/8\Z. &
\end{cases}
\]
The examples listed in Remark~\ref{rem-examples} show that these cases all occur.
In conjunction with Propositions~\ref{prop-13div}, \ref{prop-7div}, \ref{prop-5div} and Corollary \ref{cor-3div}, this proves the first statement in Theorem~\ref{thm-main}.

\section{Explicit parameterizations for each torsion structure}\label{sec-param}

In this section we complete the proof of Theorem~\ref{thm-main} by giving an explicit description of the sets
\[
S_T:=\{j(E):E(\Q(3^\infty))_\tor\simeq T\},
\]
where $T$ ranges over the set $\mathcal{T}$ of 20 possible torsion structures for $E(\Q(3^\infty))$ determined in the previous section.
It follows from Proposition~\ref{prop-twist} that the sets $S_T$ partition $\Q\backslash \{1728\}$.
As noted in Remark~\ref{rem-twist}, the $j$-invariant 1728 lies in two of the sets $S_T$, namely, the sets for $T=\Z/2\Z\oplus\Z/2\Z$ and $T=\Z/4\Z\oplus\Z/4\Z$.

We determine the sets $S_T$ in terms of sets $F_T$ of (possibly constant) rationals functions $j(t)$ that parameterize the $j$-invariants $j(E)$ of elliptic curves $E/\Q$ for which $E(\Q(3^\infty))_\tor$ contains a subgroup isomorphic to $T$.
These appear in Table~\ref{table-param} on the next page, which lists a set $F_T$ of functions $j(t)$ for each $T\in\mathcal{T}$.
Let us partially order the set $\mathcal{T}$ by inclusion (so $T_1\le T_2$ whenever $T_1$ is isomorphic to a subgroup of $T_2$).

\begin{theorem}\label{thm-param}
Let $E/\Q$ be an elliptic curve with $j(E)\ne 1728$.
Let $\mathcal{T}(E)\subseteq \mathcal{T}$ be the set of groups $T$ for which $j(E)$ lies in the image of some $j(t)\in F_{T}$.
Then $\mathcal{T}(E)$ contains a unique maximal element $T(E)$, and it is isomorphic to $E(\Q(3^\infty))_\tor$;  equivalently, $j(E)\in S_T$ if and only if $T=T(E)$.
\end{theorem}

\begin{remark}
The set $\mathcal{T}(E)$ need not contain every $T\le T(E)$. The curve \href{http://www.lmfdb.org/EllipticCurve/Q/15a1}{\texttt{15a1}} is an example: $T(E)=\Z/8\Z\oplus\Z/8\Z$ but $j(E)$ is not in the image of the unique function $j(t)$ for $T=\Z/2\Z\oplus\Z/8\Z$.
\end{remark}

\begin{corollary}\label{cor-param}
Of the $20$ groups $T$ listed in Theorem \ref{thm-main}, the following $4$ arise as $E(\Q(3^\infty))_\tor$ for only a finite set of $\Qbar$-isomorphism classes of elliptic curves $E/\Q$:
\[
\Z/4\Z\times\Z/28\Z,\qquad \Z/6\Z\times\Z/30\Z,\qquad \Z/6\Z\times\Z/42\Z,\qquad\Z/14\Z\times\Z/14\Z.
\]
The remaining $16$ arise for infinitely many $\Qbar$-isomorphism classes of elliptic curves $E/\Q$.
\end{corollary}

\begin{table}
\begin{center}
\setlength\extrarowheight{5pt}
\begin{tabular}{ll}
$T$ & $j(t)$\\\midrule
$\Z/2\Z\oplus\Z/2\Z$ & $t$\\
$\Z/2\Z\oplus\Z/4\Z$ & $\frac{(t^2+16t+16)^3}{t(t+16)}$ \\ 
$\Z/2\Z\oplus\Z/8\Z$ & $\frac{(t^4-16t^2+16)^3}{t^2(t^2-16)}$ \\ 
$\Z/2\Z\oplus\Z/10\Z$ & $\frac{(t^4-12t^3+14t^2+12t+1)^3}{t^5(t^2-11t-1)}$ \\ 
$\Z/2\Z\oplus\Z/14\Z$ & $\frac{(t^2+13t+49)(t^2+5t+1)^3}{t}$ \\ 
$\Z/2\Z\oplus\Z/16\Z$ & $\frac{(t^{16}-8t^{14}+12t^{12}+8t^{10}-10t^8+8t^6+12t^4-8t^2+1)^3}{t^{16}(t^4-6t^2+1)(t^2+1)^2(t^2-1)^4}$ \\ 
$\Z/2\Z\oplus\Z/26\Z$ & $\frac{(t^4-t^3+5t^2+t+1)(t^8-5t^7+7t^6-5t^5+5t^3+7t^2+5t+1)^3}{t^{13}(t^2-3t-1)}$ \\ 
$\Z/4\Z\oplus\Z/4\Z$ & $\frac{(t^2+192)^3}{(t^2 - 64)^2}$ \\ 
& $\frac{-16(t^4-14t^2+1)^3}{t^2(t^2+1)^4}$ \\ 
& $\frac{-4(t^2 + 2t - 2)^3(t^2 + 10t - 2)}{t^4}$ \\ 
$\Z/4\Z\oplus\Z/8\Z$ & $\frac{16(t^4+4t^3+20t^2+32t+16)^3}{t^4(t+1)^2(t+2)^4}$ \\
& $\frac{-4(t^8-60t^6+134t^4-60t^2+1)^3}{t^2(t^2-1)^2(t^2+1)^8}$ \\
$\Z/4\Z\oplus\Z/16\Z$ & $\frac{(t^{16}-8t^{14}+12t^{12}+8t^{10}+230t^8+8t^6+12t^4-8t^2+1)^3}{t^8(t^2-1)^8(t^2+1)^4(t^4- 6t^2 + 1)^2}$ \\ 
$\Z/4\Z\oplus\Z/28\Z$ & $\{\frac{351}{4},\frac{-38575685889}{16384}\}$ \\ 
$\Z/6\Z\oplus\Z/6\Z$ & $\frac{(t+27)(t+3)^3}{t}$ \\ 
$\Z/6\Z\oplus\Z/12\Z$ & $\frac{(t^2-3)^3(t^6-9t^4+3t^2-3)^3}{t^4(t^2-9)(t^2-1)^3}$ \\ 
$\Z/6\Z\oplus\Z/18\Z$ & $\frac{(t+3)^3(t^3+9t^2+27t+3)^3}{t(t^2+9t+27)}$ \\ 
& $\frac{(t+3)(t^2-3t+9)(t^3+3)^3}{t^3}$ \\ 
$\Z/6\Z\oplus\Z/30\Z$ & $\{\frac{-121945}{32},\frac{46969655}{32768}\}$ \\ 
$\Z/6\Z\oplus\Z/42\Z$ & $\{\frac{3375}{2},\frac{-140625}{8},\frac{-1159088625}{2097152},\frac{-189613868625}{128}\}$ \\ 
$\Z/8\Z\oplus\Z/8\Z$ & $\frac{(t^8+224t^4+256)^3}{t^4(t^4-16)^4}$ \\ 
$\Z/12\Z\oplus\Z/12\Z$ & $\frac{(t^2+3)^3(t^6-15t^4+75t^2+3)^3}{t^2(t^2-9)^2(t^2-1)^6}$ \\ 
& $\{\frac{-35937}{4},\frac{109503}{64}\}$ \\ 
$\Z/14\Z\oplus\Z/14\Z$ & $\{\frac{2268945}{128}\}$ \\ 
$\Z/18\Z\oplus\Z/18\Z$ & $\frac{27t^3(8-t^3)^3}{(t^3+1)^3}$ \\ 
 & $\frac{432t(t^2-9)(t^2+3)^3(t^3-9t+12)^3(t^3+9t^2+27t+3)^3(5t^3-9t^2-9t-3)^3}{(t^3-3t^2-9t+3)^9(t^3+3t^2-9t-3)^3}$ \\ 
\bottomrule
\end{tabular}
\bigskip

\caption{Parameterizations $j(t)$ of the $\Qbar$-isomorphism classes of elliptic curves $E/\Q$ according to the isomorphism type of $E(\Q(3^\infty))$.}\label{table-param}
\end{center}
\end{table}

\begin{proof}[Proof of Theorem~\ref{thm-param}]
For each group $T\in \mathcal{T}$ we enumerate subgroups $G$ of $\GL_2(\Z/n\Z)$, where $n$ is the exponent of $T$, and determine the $G$ that are maximal with respect to the following properties:
\begin{enumerate}[(i)]
\item the determinant map $G\to (\Z/n\Z)^\times$ is surjective and $G$ contains an element of trace $0$ and determinant $-1$ that acts trivially on a maximal cyclic $\Z/n\Z$-submodule of $\Z/n\Z\oplus\Z/n\Z$;
\item the submodule of $\Z/n\Z\oplus\Z/n\Z$ on which the minimal normal subgroup $N$ of $G$ for which $G/N$ is of generalized $S_3$-type acts trivially is isomorphic to $T$.
\end{enumerate}

Note that the minimal $N$ is unique, since if $N_1$ and $N_2$ are two normal subgroups of $G$ for which $G/N_1$ and $G/N_2$ are both of generalized $S_3$-type, then for $N=N_1\cap N_2$ the quotient $G/N$ is isomorphic to a subgroup of the direct product of $G/N_1$ and $G/N_2$, hence also of generalized $S_3$-type.
We recall that (i) is necessarily satisfied by any subgroup $G$ of $\GL_2(\Z/n\Z)$ that arises as the image of $\rho_{E,n}$ for an elliptic curve $E/\Q$, and (ii) implies that if $G\simeq\rho_{E,n}(\Gal(\Q(E[n])/\Q))$ for some $E/\Q$, then $G/N\simeq \Gal((\Q(E[n])\cap\Q(3^\infty))/\Q)$ and $N\simeq \Gal(\Q(E[n])/(\Q(E[n])\cap\Q(3^\infty))$.
The $n$-torsion points of~$E$ fixed by $\Gal(\Qbar/\Q(3^\infty))$ must then form a subgroup isomorphic to $T$, equivalently, $E(\Q(3^\infty))_\tor$ contains a subgroup isomorphic to $T$.
The existence of the examples in Remark~\ref{rem-examples} ensures that we get at least one maximal $G$ for each $T$.

Our maximality condition ensures that $G$ always contains $-1$ (otherwise we can add $-1$ to both $G$ and $N$).
The corresponding modular curve $X_G$ has a rational model (because the determinant map of $G$ is surjective), and each non-cuspidal rational point on $X_G$ determines a $\Qbar$-isomorphism class that contains an elliptic curve $E/\Q$ for which $\im \rho_{E,n}$ is conjugate in $\GL_2(\Z/n\Z)$ to a subgroup of $G$.
For $j(E)\ne 1728$ the group $E(\Q(3^\infty))_\tor$ depends only on $j(E)$, by Proposition~\ref{prop-twist}, thus we may restrict our attention to the image $J_G$ of the non-cuspidal points in $X_G(\Q)$ under the map to $X(1)$; if $j(E)$ lies in this image then there is an elliptic curve $E'$ in this $\Qbar$-isomorphism class for which $\im\rho_{E',n}$ is conjugate to a subgroup of $G$, and it follows that $E'(\Q(3^\infty))_\tor$, and therefore $E(\Q(3^\infty))_\tor$, must contain a subgroup isomorphic to $T$.
In the other direction, if $E(\Q(3^\infty))_\tor\simeq T$, then $\im\rho_{E,n}$ must be conjugate to a subgroup of one of the maximal groups $G$ for this $T$, and $j(E)$ must lie in~$J_G$.
The set $\mathcal{T}(E)$ thus contains a unique maximal element, namely, $T(E)\simeq E(\Q(3^\infty))_\tor$, since if $T'\in \mathcal{T}(E)$ then $E(\Q(3^\infty))_\tor\simeq T$ must contain a subgroup isomorphic to $T'$.
The theorem then follows, provided that for each $T\in\mathcal{T}$ we can determine a set of rational functions $F_T$ for which the union of the images of these functions is equal to the union of the image $J_G$ over the maximal groups $G$ for $T$.
This amounts to explicitly expressing each of the images $J_G$ as the union of the images of a set of (possibly constant) rational functions $j(t)$.  We turn now to this problem.

We first note that it may happen that $G$ is the full inverse image of the reduction map from $\GL_2(\Z/n\Z)$ to $\GL_2(\Z/m\Z)$ for some $m$ dividing $n$; in this case we reduce $G$ modulo the largest such $m$ and call $m$ the \emph{level} of $G$. For example, when $T=\Z/2\Z\oplus\Z/2\Z$ we have $G=\GL_2(\Z/2\Z)$ and can reduce $G$ to the trivial group of level 1 corresponding to $X(1)$; this is consistent with the fact that $E(\Q(3^\infty))[2]=E[2]$ holds for all $E/\Q$.  Similar remarks apply whenever $n=2m$ with $m$ odd.

A Magma script to enumerate the maximal groups $G$ for each torsion structure $T$ can be found at~\cite{magmascripts};
for each $G$ we may determine the genus of $X_G$ by taking the intersection of $G$ with $\SL_2(\Z/n\Z)$ (all the cases of interest are already listed in the tables of Cummins and Pauli \cite{cp}), and we use \cite[Lemma 3.4]{zywina2} to determine the number of rational cusps on $X_G$.
There are a total of 33 maximal groups $G$ for the 20 groups $T$, and we find that for each of these $G$, one of the following holds: (1) $X_G$ has genus 0 and a rational point, in which case $X_G$ is isomorphic to $\mathbb{P}^1$ and the map $X_G\to X(1)$ is given by a rational function $j(t)$, or (2) $X_G$ is isomorphic to either a genus 1 curve with no rational points, an elliptic curve of rank 0, or a curve of genus greater than 1, and in every case the image of $X_G(\Q)$ in $X(1)$ is finite (by Faltings' Theorem \cite{faltings}).

For the first five groups $T$ listed in Table~\ref{thm-param}, there is a unique maximal $G$ and $X_G$ has genus 0 and is of prime-power level; for these $G$ we may take $j(t)$ from \cite{SZ} (for the 2-power levels, maps that are equivalent up to an automorphism of $\mathbb{P}^1$ (hence have the same image) can also be found in the tables of~\cite{RZB}).
The same applies to the groups $\Z/2\Z\oplus\Z/26\Z$, $\Z/6\Z\oplus\Z/6\Z$, and $\Z/8\Z\oplus\Z/8\Z$. We now briefly discuss each of the remaining 12 groups $T$:

\begin{itemize}
\setlength{\itemsep}{4pt}
\item $\Z/2\Z\oplus\Z/16\Z$: There are two maximal $G$, both of level 16; for the first, $X_G$ has genus 0 and the corresponding map $j(t)$ from \cite{SZ} is listed in Table~\ref{table-param}.
For the second $X_G$ is a genus 1 curve with no rational points (the curve \texttt{X335} in \cite{RZB}).

\item $\Z/4\Z\oplus\Z/4\Z$: There are three maximal $G$, one of level 2 and two of level 4, all of genus 0; the corresponding maps $j(t)$ from \cite{SZ} are listed in Table~\ref{table-param}.

\item $\Z/4\Z\oplus\Z/8\Z$: There are two maximal $G$, one of level 4 and one of level 8, both of genus 0; the corresponding maps $j(t)$ from \cite{SZ} are listed in Table~\ref{table-param}.

\item $\Z/4\Z\oplus\Z/16\Z$: There are two maximal $G$, one of level 8 and one of level 16. The level 8 curve has genus 0 and the corresponding map $j(t)$ from \cite{SZ} is listed in Table~\ref{table-param}, while the level 16 curve is a genus 1 curve with no rational points (the curve \texttt{X478} in \cite{RZB}).

\item $\Z/4\Z\oplus\Z/28\Z$: There are three maximal $G$, one of level 14 and two of level 28, all of which have genus greater than 2. Two are ruled out by the fact that any $E/\Q$ with this image would be isogenous to an $E'/\Q$ admitting a rational 28-isogeny, but no such $E'$ exist, by Theorem~\ref{thm-isog}.
The remaining $G$ of level 28 corresponds to a modular curve $X_G$ of of genus~3 with congruence subgroup
\href{http://www.uncg.edu/mat/faculty/pauli/congruence/csg3.html\#group28E3}{$\textrm{28E}^3$}.
This curve admits a degree-2 map to a genus~2 curve $X_H$, where $G\subseteq H$, with congruence subgroup
\href{http://www.uncg.edu/mat/faculty/pauli/congruence/csg2.html\#group28A2}{$\textrm{28A}^2$}.
The curve $X_H$ has a hyperelliptic model
\[
X_H: y^2 = x^6 - 2x^5 - 4x^4 - 4x^3 - 4x^2 - 2x + 1
\]
whose Jacobian has rank 1.
Chabauty's method finds that $X_H$ has 4 rational points, two of which are the image of known non-cuspidal rational points on $X_G$ (the corresponding $j$-invariants are listed in Table~\ref{table-param}), while the other two are cusps.

\item $\Z/6\Z\oplus\Z/12\Z$: There is one maximal $G$ and it is conjugate to the Borel group in $\GL_2(\Z/12\Z)$, and $X_G=X_0(12)$ has genus 0; the map to the $j$-line is taken from \cite[Table 3]{lozano1}.

\item $\Z/6\Z\oplus\Z/18\Z$: There are three maximal $G$, all of level 9, two of genus 0 and one of~genus~1.
The corresponding maps $j(t)$ for the genus 0 curves form \cite{SZ} are listed in Table~\ref{table-param}.
As shown in the proof of \ref{lem-3x9},
the genus~1 curve has only one non-cuspidal rational point corresponding to
$j$-invariant 0, but for $j(E)=0$ we have $E(\Q(3^\infty))_\tor\simeq \Z/18\Z\oplus\Z/18\Z$.

\item $\Z/6\Z\oplus\Z/30\Z$: There is one maximal $G$, of level 15 and genus 1 and $X_G$ admits a map to $X_0(15)$  whose rational points give four distinct $j$-invariants; see ~\cite[Table 4]{lozano1}.
Of these, two correspond to elliptic curves whose mod-15 Galois image is isomorphic to a subgroup of $G$ (of index 2 but yielding the same $E(\Q(3^\infty))_\tor$ structure); these are listed in Table~\ref{table-param}.

\item $\Z/6\Z\oplus\Z/42\Z$: There is one maximal $G$, of level 21 and genus 1, and $X_G$ is the curve $X_0(21)$ whose rational points give rise to four the $j$-invariants listed in Table~\ref{table-param}; see \cite[Table 4]{lozano1}.

\item $\Z/12\Z\oplus\Z/12\Z$: There are three maximal $G$, one of level 6 and genus 0 whose corresponding map $j(t)$ can be computed as a fiber product of maps in \cite{SZ}; this map appears in Table~\ref{table-param}.
The other two have level 12 and genus 1, and the $X_G$ are isomorphic to \href{http://www.lmfdb.org/EllipticCurve/Q/48a1}{\texttt{48a1}} and \href{http://www.lmfdb.org/EllipticCurve/Q/48a3}{\texttt{48a3}} respectively.
The first has four rational points, all cuspidal, and the second has eight rational points, four of which are non-cuspidal and yield the two $j$-invariants listed in Table~\ref{table-param}.

\item $\Z/18\Z\oplus\Z/18\Z$: There are two maximal $G$, one of level 3 and one of level 9 and both of genus 0; the corresponding maps $j(t)$ from \cite{SZ} appear in Table~\ref{table-param}.
\end{itemize}
Further details of these computations can be found in \cite{magmascripts}.
\end{proof}

\end{document}